%% file: main.tex
\documentclass{ifacconf}

\let\theoremstyle\relax

\input{latex_common/preamble.tex}
\input{latex_common/commands.tex}

\input{latex_common/style.tex}
\input{special.tex}

\usepackage{natbib}

\graphicspath{{./figures/}}

\begin{document}

\begin{frontmatter}
  \title{Lossless Convexification of Optimal Control Problems with
    Semi-continuous Inputs}

  
  \author[Author]{Danylo Malyuta}
  \author[Author]{\Behcet \Acikmese}

  \address[Author]{Dept. of Aeronautics \& Astronautics,
    University of Washington, Seattle, WA 98195 USA (e-mail:
    \texttt{\{danylo,behcet\}@uw.edu})}

  \begin{abstract}                
    This paper presents a novel convex optimization-based method for finding the
    globally optimal solutions of a class of mixed-integer non-convex optimal
    control problems. We consider problems with non-convex constraints that
    restrict the input norms to be either zero or lower- and upper-bounded. The
    non-convex problem is relaxed to a convex one whose optimal solution is
    proved to be optimal almost everywhere for the original problem, a procedure
    known as lossless convexification. This paper is the first to allow
    individual input sets to overlap and to have different norm bounds, integral
    input and state costs, and convex state constraints that can be activated at
    discrete time instances. The solution relies on second-order cone
    programming and demonstrates that a meaningful class of optimal control
    problems with binary variables can be solved reliably and in polynomial
    time. A rocket landing example with a coupled thrust-gimbal constraint
    corroborates the effectiveness of the approach.
  \end{abstract}

\end{frontmatter}

\section{Introduction}

We present a convex programming solution to a class of optimal control problems
with semi-continuous control input norms. Semi-continuous variables are a
particular type of binary non-convexity.

\begin{definition}
  \label{definition:semicontinuous}
  Variable $x\in\reals$ is \textit{semi-continuous} if $x\in \{0\}\union[a,b]$
  with $0<a\le b$ \cite{MosekCookbook2019}.
\end{definition}

The constraint $az\le x\le bz$ with $z\in\{0,1\}$ models
semi-continuity. Practical rocket landing and spacecraft rendezvous path
planning problems include such constraints, and can take hours to solve using
existing mixed-integer convex programming (MICP) methods. In this paper, we
propose an algorithm based on lossless convexification that solves these
problems to global optimality in seconds.

Lossless convexification is a method for finding the globally optimal solution
of non-convex problems using convex optimization. The method relaxes the
original problem to a convex one via a slack variable, enabling the use of
second-order cone programming (SOCP). The maximum principle is used to prove
that the solution of the relaxed problem is globally optimal for the original
problem.

Classical lossless convexification deals with non-convexity in the form of an
input norm lower-bound. The first result was introduced in \cite{Acikmese2007}
for minimum-fuel rocket landing and was later expanded to more general
non-convex input sets \cite{Acikmese2011}. Extensions of the method were
introduced in \cite{Blackmore2010,
  Carson2011,Acikmese2013} to handle minimum-error rocket landing and non-convex
pointing constraints. More recently, lossless convexification was shown to
handle affine and quadratic state constraints \cite{Harris2013a,Harris2013b},
culminating in \cite{Harris2014}.

A recurring assumption of classical lossless convexification is that there is a
single input which cannot be turned off. Our interest is in problems that have
multiple inputs which may be turned off. When active, the input norm is
lower-bounded, making it semi-continuous in the sense of
Definition~\ref{definition:semicontinuous}. This is a richer binary
non-convexity than what was handled by classical lossless convexification.

The concept of lossless convexification with binary variables implemented via
MICP was explored in \cite{Blackmore2012,Zhang2017}. However, the \NPhard nature
of MICP generally makes the approach computationally expensive. Recently, a
limited class of binary non-convexity was handled via lossless convexification
in \cite{Malyuta2019}, proving that a class of \NPhard problems is of \PP
complexity. The approach is amenable to real-time onboard optimization for
autonomous systems and for rapid design trade studies.

Our main contribution is to extend the lossless convexification result of
\cite{Malyuta2019}. The list of extensions that we introduce is as follows. We
allow an input integral cost, a state integral cost, different norm lower- and
upper-bounds for each input, overlapping pointing directions of the inputs, and
state constraints.

The paper is organized as follows. Section~\ref{section:problem_statement}
defines the class of optimal control problems that our method
handles. Section~\ref{section:lcvx} proposes our solution method based on
lossless convexification. Section~\ref{section:lcvx_proof} proves that our
method finds the globally optimal solution based on the necessary conditions of
optimality presented in
Section~\ref{section:maximum_principle}. Section~\ref{section:example} presents
a rocket landing example which corroborates the method's effectiveness for
practical path planning applications. Section~\ref{section:future_work} outlines
future work and Section~\ref{section:conclusion} summarizes the result.


\textit{Notation}: sets are calligraphic, e.g. $\mathcal S$. Set $\reals^n_{-}$
denotes the $n$-dimensional non-positive orthant. The operator $\circ$ denotes
the element-wise product. Given a function
$f:\reals^n\times\reals^m\to\reals^p$, we use the shorthand
$f[t]\equiv f(x(t),y(t))$. In text, functions are referred to by their letter
(e.g. $f$) and conflicts with another variable are to be understood from
context. The gradient of $f$ with respect an argument $x$ is denoted
$\grad_x f\in\reals^{p\times n}$. Similarly, if $f$ is nonsmooth then its
subdifferential with respect to $x$ is
$\subdiff[x] f\subseteq\reals^{1\times n}$. The normal cone at $x$ to
$\mathcal S\subseteq\reals^n$ is denoted
$\normalcone{\mathcal S}{x}\subseteq\reals^n$. When we refer to an
\textit{interval}, we mean some time interval $[t_1,t_2]$ of non-zero duration,
i.e. $t_1<t_2$. We call the Eucledian projection of $y\in\reals^n$ onto
$\set S\subseteq\reals^n$ the magnitude of the 2-norm projection of $y$:
\begin{eqnarray}
  \label{eq:eucledian_projection}
  \Proj{\set S}{y}\definedas\big\|\textstyle\argmin_{z\in\set S}\norm{y-z}{2}\big\|_2.
\end{eqnarray}

\section{Problem Definition}
\label{section:problem_statement}

This section presents the class of optimal control problems that can be solved
via convex optimization by our method. We consider mixed-integer non-convex
optimal control problems with linear time-invariant (LTI) dynamics and
semi-continuous input norms:
\begin{poptimization}{center}{ocp}{}{u_i,\gamma_i,t_f}{%
    \begin{array}{l}
      m(t_f,x(t_f))+\int_{0}^{t_f}\ell(x(t))+ \\
      \runningku\inlinesum_{i=1}^M\|u_i(t)\|_2\dd t
    \end{array}}{$\mathcal O$}
  \dot x(t) = Ax(t)+B\inlinesum_{i=1}^M{u_i(t)}+w,~x(0)=x_0, \#
  \gamma_i(t)\rho_1^i \le \|u_i(t)\|_2\le \gamma_i(t)\rho_2^i\quad
  i=1,\dots,M, \#
  \gamma_i(t)\in \{0,1\}\quad i=1,\dots,M, \#
  \inlinesum_{i=1}^M\gamma_i(t) \le K, \#
  C_iu_i(t)\le 0\quad i=1,\dots,M, \#
  x(t)\in\mathcal X, \#
  b(x(t_f)) = 0,
\end{poptimization}
where $x(t)\in\reals^n$ is the state, $u_i(t)\in\reals^m$ is the $i$-th input,
and $w\in\real^n$ is a known external input. Convex functions
$m:\reals\times\reals^n\to\reals$, $\ell:\reals^n\to\reals$ and
$b:\reals^n\to\reals^{n_b}$ define the terminal cost, the state running cost and
the terminal manifold respectively. The binary coefficient
$\runningku\in\{0,1\}$ toggles the input running cost. The state must lie in the
convex set $\mathcal X\subseteq\reals^n$. The input directions are constrained
to polytopic cones called \textit{input pointing sets}:
\begin{equation}
  \label{eq:input_pointing_set}
  \mathcal U_i\definedas \{u\in\reals^m: C_iu\le 0\},
\end{equation}
where $C_i\in\reals^{p_i\times m}$ is a matrix with $C_{i,j}$ the $j$-th row.

\begin{assumption}
  \label{ass:full_range_cone}

  Matrices $C_i$ in \eqref{eq:ocp_f} are full row rank.
\end{assumption}

\begin{assumption}
  \label{ass:norm_bounds}
  The control norm bounds in \eqref{eq:ocp_c} are distinct,
  i.e. $\rho_1^i<\rho_2^i$.
\end{assumption}

Problem~\ref{problem:ocp} extends the problem class in \cite{Malyuta2019} in
several non-trivial ways. First, there are input and state integral costs in
\eqref{eq:ocp_a}. Second, the input norm bounds in \eqref{eq:ocp_c} can be
different for each input. Third, the state can be constrained to a convex set in
\eqref{eq:ocp_g}. Last and most important, \cite[Assumption~1]{Malyuta2019} is
removed, such that the input pointing sets can overlap
arbitrarily. Figure~\ref{fig:input_set_overlap} shows how this enables richer
input set geometry than permitted in \cite{Malyuta2019}.

\begin{figure}
  \centering
  \begin{subfigure}[b]{0.51\columnwidth}
    \centering
    \includegraphics[width=0.923\textwidth]{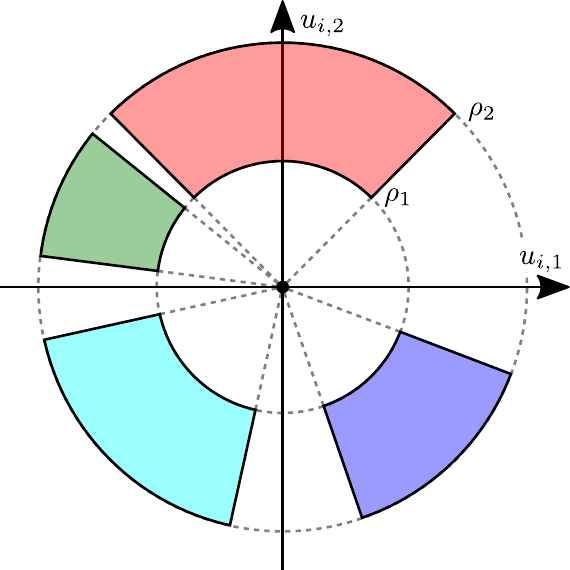}
    \caption{Disjoint same-bound input sets allowed by \cite{Malyuta2019}.}
    \label{fig:input_set_overlap_old}
  \end{subfigure}
  \hfill
  \begin{subfigure}[b]{0.47\columnwidth}
    \centering
    \includegraphics[width=1\textwidth]{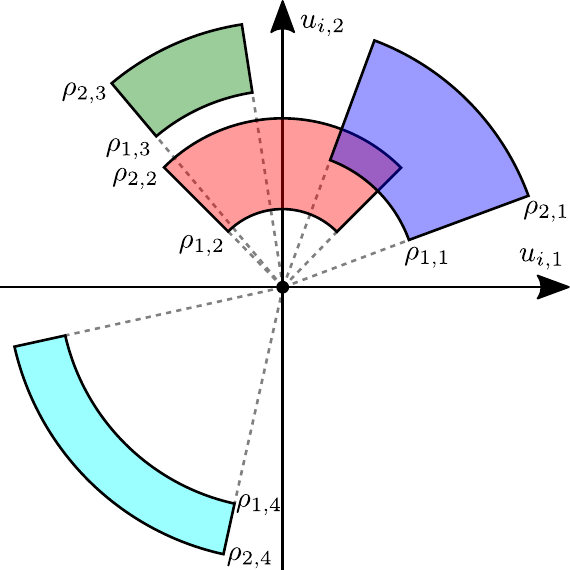}
    \caption{Overlapping multiple-bound input sets allowed in this paper.}
    \label{fig:input_set_overlap_new}
  \end{subfigure}
  \caption{Input set example for $m=2$. As shown in
    (\protect\subref{fig:input_set_overlap_old}), \cite{Malyuta2019} does not
    allow different input norm bounds or overlapping input pointing sets. As
    shown in (\protect\subref{fig:input_set_overlap_new}), this paper allows
    both features.}
  \label{fig:input_set_overlap}
\end{figure}

\section{Lossless Convexification}
\label{section:lcvx}

This section presents the two main results, Theorems~\ref{theorem:lcvx_a} and
\ref{theorem:lcvx_b}, which state that the convex Problem~\ref{problem:rcp}
finds the global optimum of Problem~\ref{problem:ocp} under certain conditions.

The input magnitude in Problem~\ref{problem:ocp} is semi-continuous, i.e.
$\|u_i(t)\|_2\in\{0\}\union[\rho_1^i,\rho_2^i]$. This makes the problem
mixed-integer and non-convex, which is readily apparent from
Figure~\ref{fig:input_set_overlap}. Consider the following convex relaxation:
\begin{poptimization}{center}{rcp}{}{u_i,\gamma_i,\sigma_i,t_f}{%
    \begin{array}{l}
      m(t_f,x(t_f))+\runningku\xi(t_f)+ \\
      \int_{0}^{t_f}\ell(x(t))\dd t
    \end{array}}{$\mathcal R$}
  \dot x(t) = Ax(t)+B\inlinesum_{i=1}^M{u_i(t)}+w,~x(0)=x_0, \#
  \dot \xi(t) = \inlinesum_{i=1}^M\sigma_i(t), \#
  \gamma_i(t)\rho_1^i \le \sigma_i(t)\le \gamma_i(t)\rho_2^i\quad
  i=1,\dots,M, \#
  \|u_i(t)\|_2\le\sigma_i(t)\quad i=1,\dots,M, \#
  0\le \gamma_i(t)\le 1\quad i=1,\dots,M, \#
  \inlinesum_{i=1}^M\gamma_i(t) \le K, \#
  C_iu_i(t)\le 0\quad i=1,\dots,M, \#
  x(t)\in\mathcal X, \#
  b(x(t_f)) = 0.
\end{poptimization}

Replacing \eqref{eq:ocp_c}-\eqref{eq:ocp_d} with
\eqref{eq:rcp_d}-\eqref{eq:rcp_f} convexifies the input set of
Problem~\ref{eq:ocp}. Figure~\ref{fig:input_set_relaxation} illustrates an
example.

\begin{figure*}
  \centering
  \begin{subfigure}[b]{0.32\textwidth}
    \centering
    \includegraphics[width=0.75\columnwidth]{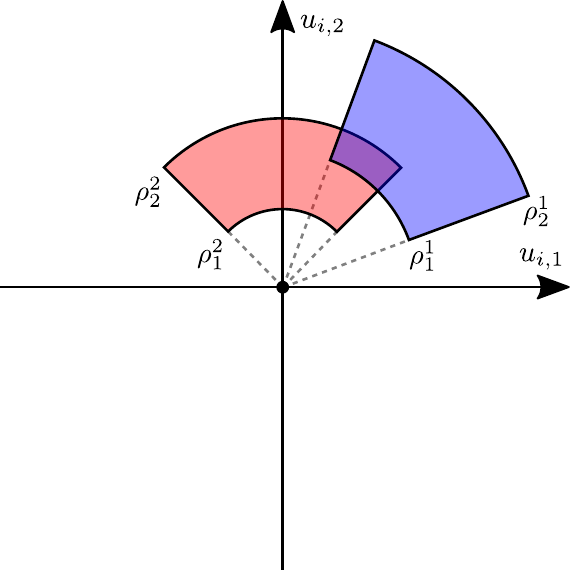}
    \caption{Original non-convex input set defined by
      \eqref{eq:ocp_c}-\eqref{eq:ocp_f}.}
    \label{fig:input_set_relaxation_1}
  \end{subfigure}%
  \hfill%
  \begin{subfigure}[b]{0.32\textwidth}
    \centering
    \includegraphics[width=1\columnwidth]{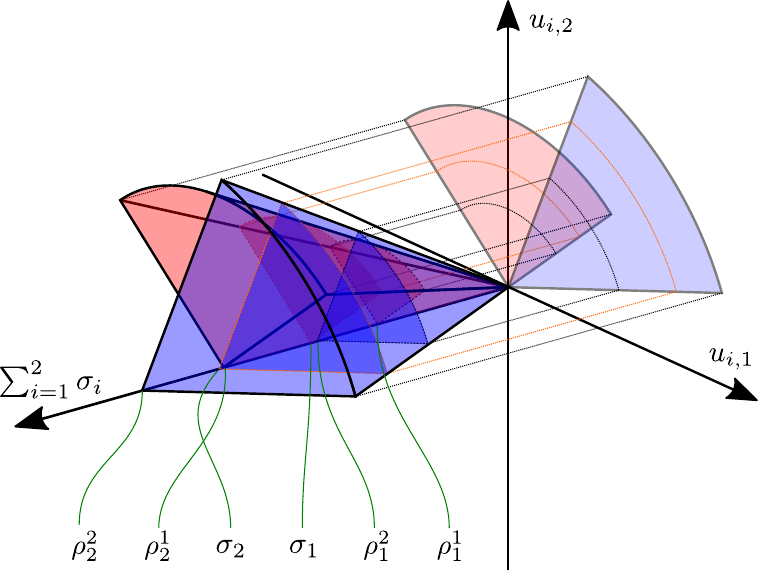}
    \caption{Non-convexity of individual input sets is removed by relaxing
      \eqref{eq:ocp_c} to \eqref{eq:rcp_d}-\eqref{eq:rcp_e}.}
    \label{fig:input_set_relaxation_2}
  \end{subfigure}%
  \hfill%
  \begin{subfigure}[b]{0.32\textwidth}
    \centering
    \includegraphics[width=1\columnwidth]{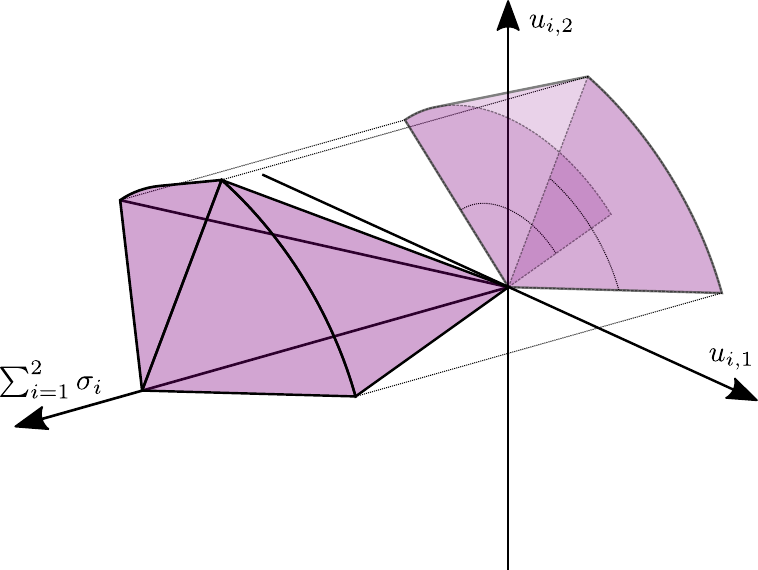}
    \caption{Semi-continuity of the input norm is convexified by relaxing
      \eqref{eq:ocp_d} to \eqref{eq:rcp_f}.}
    \label{fig:input_set_relaxation_3}
  \end{subfigure}
  \caption{Problem~\ref{problem:rcp} convexifies the input set of
    Problem~\ref{problem:ocp}, here shown for $M=2$, $K=1$ and $m=2$. The
    relaxation consists of three steps:
    \protect\subref{fig:input_set_relaxation_1})
    \eqref{eq:ocp_c}-\eqref{eq:ocp_f} originally define a non-convex set of a
    binary nature; \protect\subref{fig:input_set_relaxation_2}) by relaxing
    \eqref{eq:ocp_c} to \eqref{eq:rcp_d}-\eqref{eq:rcp_e}, individual input sets
    are convexified; \protect\subref{fig:input_set_relaxation_3}) by relaxing
    \eqref{eq:ocp_d} to \eqref{eq:rcp_f}, a convex hull is obtained.}
  \label{fig:input_set_relaxation}
\end{figure*}

Consider the following conditions which remove degenerate solutions of
Problem~\ref{problem:rcp} that may be infeasible for
Problem~\ref{problem:ocp}. The conditions use an \textit{adjoint system} whose
output $y(t)\in\reals^m$ is called the \textit{primer vector}:
\begin{subequations}
  \label{eq:adjoint_system}
  \begin{align}
    \dot\lambda(t) &= -A\T\lambda(t)+v(t),~v(t)\in\subdiff\ell(x(t))\T, \label{eq:adjoint_dynamics} \\
    y(t) &= B\T\lambda(t). \label{eq:primer_vector}
  \end{align}
\end{subequations}

It will be seen in Section~\ref{section:lcvx_proof} that we are interested in
``how much'' $y(t)$ projects onto the $i$-th input pointing set. This is given
by the following input \textit{gain} measure:
\begin{equation}
  \label{eq:input_gain}
  \Gamma_i(t) \definedas (\Proj{\set{U_i}}{y(t)}-\runningku)\rho_2^i.
\end{equation}

\begin{condition}{}
  \label{condition:observability}
  The adjoint system \eqref{eq:adjoint_system} is strongly observable
  \cite{Trentelman2001}.
\end{condition}

\begin{condition}{}
  \label{condition:normality}
  The adjoint system \eqref{eq:adjoint_system} and pointing cone geometry
  \eqref{eq:ocp_f} satisfy either:
  \begin{enumerate}
  \item[(a)] $\Gamma_i(t)\ne 0~\alev{[0,t_f]}$ $\forall i$
    s.t. $y(t)\notin\interior{\normalcone{\set U_i}{0}}$;
  \item[(b)] on any interval where $\Gamma_i(t)=0$, $\Gamma_j(t)>0$ for at least
    $K$ other inputs.
  \end{enumerate}
\end{condition}

\begin{condition}{}
  \label{condition:ambiguity}
  The adjoint system \eqref{eq:adjoint_system} and pointing cone geometry
  \eqref{eq:ocp_f} satisfy either:
  \begin{enumerate}
  \item[(a)] $\Gamma_i(t)\ne \Gamma_j(t)~\alev{[0,t_f]}$ $\forall i$
    s.t. $y(t)\notin\interior{\normalcone{\set U_i}{0}}$;
  \item[(b)] on any interval where $\Gamma_i(t)=\Gamma_j(t)$, there exist $K$
    inputs with $\Gamma_k(t)>\Gamma_i(t)$ or $M-K$ inputs where
    $\Gamma_k(t)<\Gamma_i(t)$.
  \end{enumerate}
\end{condition}

\begin{condition}{}
  \label{condition:transversality}
  $\ell[t]+\runningku\sum_{i=1}^M\sigma_i(t)+\grad_t m[t_f]\ne 0$
  $\forall t\in [0,t_f)$.
\end{condition}

We now state the two main results of this paper, which claim that
Problem~\ref{problem:rcp} solves Problem~\ref{problem:ocp} under certain
conditions. The theorems are proved in Section~\ref{section:lcvx_proof}.

\begin{theorem}{a}
  \label{theorem:lcvx_a}
  The solution of Problem~\ref{problem:rcp} is globally optimal $\alev{[0,t_f]}$
  for Problem~\ref{problem:ocp} if
  Conditions~\ref{condition:observability}-\ref{condition:transversality} hold
  and the state constraint \eqref{eq:ocp_g} is never activated.
\end{theorem}

\addtocounter{theorem}{-1}

\begin{theorem}{b}
  \label{theorem:lcvx_b}
  The solution of Problem~\ref{problem:rcp} is globally optimal $\alev{[0,t_f]}$
  for Problem~\ref{problem:ocp} if
  Conditions~\ref{condition:observability}-\ref{condition:transversality} hold
  and the state constraint \eqref{eq:ocp_g} is activated at discrete times.
\end{theorem}

\subsection{Discussion on Strong Observability}
\label{subsec:discussion_observability}

This section describes Condition~\ref{condition:observability} and its
verification. Strong observability extends the concept of observability to the
case of non-zero inputs. A strongly observable system does not have transmission
zeroes. To be precise, let us state strong observability in the context of
\eqref{eq:adjoint_system}.

\begin{definition}[{\citealp[Definition~7.8]{Trentelman2001}}]
  A point $\lambda_0\in\reals^n$ is \textit{weakly unobservable} if there exists
  an interval $\set{T}=[\tau_1,\tau_2]$ and an input trajectory
  $v(t)\in\subdiff\ell[t]\T$ for $t\in\set{T}$ such that if
  $\lambda(\tau_1)=\lambda_0$ then the primer vector satisfies $y(t)=0$
  $\forall t\in\set{T}$. The set of all weakly unobservable points is denoted
  $\set{V}$, which is called the weakly unobservable set.
\end{definition}

\begin{theorem}{}[{\citealp[Theorem~7.16]{Trentelman2001}}]
  \label{theorem:strong_obsv}
  The adjoint system \eqref{eq:adjoint_system} is strongly observable if
  $\set{V}=\{0\}$.
\end{theorem}

To verify Condition~\ref{condition:observability} via simple matrix algebra, it
is sufficient to apply the algorithm for computing $\set{V}$ in
\cite[Section~7.3]{Trentelman2001} using the following alternative to
\eqref{eq:adjoint_dynamics}:
\begin{equation}
  \dot\lambda(t) = -A\T\lambda(t)+Dv(t),
\end{equation}
where $v(t)\in\reals^n$ and
$\range D = \linspan\Union_{t\in [\tau_1,\tau_2]}\subdiff\ell(x(t))\T$. This
conservative approximation assumes that the input can come from a subspace
spanned by the subdifferentials. Section~\ref{section:example} uses this
approximation to verify Conditions
\ref{condition:observability}-\ref{condition:ambiguity} for the rocket landing
problem.

\section{Nonsmooth Maximum Principle}
\label{section:maximum_principle}

This section states a nonsmooth version of the maximum principle that we shall
use for proving Theorems~\ref{theorem:lcvx_a} and \ref{theorem:lcvx_b}. Consider
the following general optimal control problem:
\begin{poptimization}{center}{ocp_general}{}{u,t_f}{%
    \begin{array}{l}
      m(t_f,x(t_f))+\int_{\tzero}^{t_f}\ell(t,u(t),x(t))\dd t
    \end{array}}{$\mathcal G$}
  \dot x(t) = f(t,x(t),u(t)),\quad x(\tzero)=x_0, \#
  g(t,u(t))\le 0, \#
  b(t_f,x(t_f)) = 0.
\end{poptimization}
where the state trajectory $x(\cdot)$ is absolutely continuous and the control
trajectory $u(\cdot)$ is measurable. The dynamics
$f:\reals\times\reals^n\times\reals^m\to\reals^n$ are convex and continuously
differentiable. The terminal cost $m:\reals\times\reals^n\to\reals$, the running
cost $\ell:\reals\times\reals^m\times\reals^n\to\reals$, the input constraint
$g:\reals\times\reals^m\to\reals^{n_g}$, and the terminal constraint
$b:\reals\times\reals^n\to\reals^{n_b}$ are convex. Define the terminal manifold
as
$\set T\definedas\{x\in\reals^n:\textnormal{\eqref{eq:ocp_general_d} holds}\}$
and the Hamiltonian function:
\begin{eqnarray}
  \label{eq:ocp_nonsmooth_hamiltonian}
  H(t,x(t),u(t),\alpha,\psi(t))\definedas \alpha\ell[t]+\psi(t)\T f[t],
\end{eqnarray}
where $\alpha\le 0$ is the \textit{abnormal multiplier} and $\psi(\cdot)$ is the
\textit{adjoint variable} trajectory. We now state the nonsmooth maximum
principle, due to \cite[Theorem~8.7.1]{Vinter2000} (see also
\cite{Clarke2010,Hartl1995}), which specifies the necessary conditions of
optimality for Problem~\ref{problem:ocp_general}.

\begin{theorem}{}[Maximum Principle]
  \label{theorem:mp}
  Let $x(\cdot)$ and $u(\cdot)$ be optimal on the interval $[0,t_f]$. There
  exist a constant $\alpha\le 0$ and an absolutely continuous $\psi(\cdot)$
  such that the following conditions are satisfied:
  \begin{enumerate}
  \item Non-triviality:
    \begin{eqnarray}
      \label{eq:ocp_nonsmooth_non_triviality}
      (\alpha,\psi(t))\ne 0~\forall t\in [\tzero,t_f];
    \end{eqnarray}
  \item Pointwise maximum:
    \begin{eqnarray}
      \label{eq:ocp_nonsmooth_pointwise_maximum_argmax}
      u(t) = \argmax_{v\in\textnormal{\eqref{eq:ocp_general_c}}}
      H(t,x(t),v,\alpha,\psi(t))~\alev{[0,t_f]};
    \end{eqnarray}
  \item The differential equations and inclusions:
    \begin{subequations}
      \begin{align}
        \label{eq:ocp_nonsmooth_dynamics_state}
        \dot x(t)
        &= \grad_{\psi} H[t]\T~\alev{[\tzero,t_f]}, \\
        \label{eq:ocp_nonsmooth_dynamics_adjoint}
        \dot \psi(t)
        &\in -\subdiff[x] H[t]\T~\alev{[\tzero,t_f]}, \\
        \label{eq:ocp_nonsmooth_dynamics_hamiltonian}
        \dot H[t]
        &\in \subdiff[t] H[t]~\alev{[\tzero,t_f]};
      \end{align}
    \end{subequations}
  \item Transversality:
    \begin{subequations}
      \label{eq:ocp_nonsmooth_transversality}
      \begin{align}
        \label{eq:ocp_nonsmooth_transversality_adjoint}
        \psi(t_f) &\in \alpha\subdiff[x] m[t_f]\T+\normalcone{\set T}{x(t_f)}, \\
        \label{eq:ocp_nonsmooth_transversality_hamiltonian}
        0 &\in H[t_f]+\alpha\subdiff[t] m[t_f]+\normalcone{\set T}{t_f}.
      \end{align}
    \end{subequations}
  \end{enumerate}
\end{theorem}

\section{Lossless Convexification Proof}
\label{section:lcvx_proof}

This section proves Theorems~\ref{theorem:lcvx_a} and \ref{theorem:lcvx_b}. The
general outline is as follows. We first prove Theorem~\ref{theorem:lcvx_a} by
showing that (step 1) the solution of Problem~\ref{problem:rcp} is feasible for
Problem~\ref{problem:ocp}, and (step 2) the solution is globally optimal. We
then show Theorem~\ref{theorem:lcvx_b} via a proof by contradiction in which
Theorem~\ref{theorem:lcvx_a} is applied on each interval where the state
constraint is inactive.

\begin{lemma}
  \label{lemma:lcvx}
  The solution of Problem~\ref{problem:rcp} is feasible $\alev{[0,t_f]}$ for
  Problem~\ref{problem:ocp} if $x(t)\in\interior{\set X}$ and Conditions
  \ref{condition:observability}-\ref{condition:transversality} hold.

  \begin{proof}
    The proof uses the maximum principle from Theorem~\ref{theorem:mp}. Since
    there are two states, partition the adjoint variable as
    $\psi(t)=(\lambda(t)\in\reals^n,\eta(t)\in\reals)$. For
    Problem~\ref{problem:rcp} and $x(t)\in\interior{\mathcal X}$, the adjoint
    and Hamiltonian dynamics follow from
    \eqref{eq:ocp_nonsmooth_dynamics_adjoint} and
    \eqref{eq:ocp_nonsmooth_dynamics_hamiltonian}:
    \begin{subequations}
      \label{eq:proof_dynamics}
      \begin{align}
        \dot\lambda(t) &= -A\T\lambda(t)-\alpha v(t),~v(t)\in\subdiff\ell[t]\T,
                         ~\alev{[0,t_f]}, \label{eq:lambda_dynamics} \\
                         \dot\eta(t) &= 0~\alev{[0,t_f]}, \label{eq:eta_dynamics} \\
                         \dot H[t] &= 0~\alev{[0,t_f]}, \label{eq:hamiltonian_dynamics}
      \end{align}
    \end{subequations}

    Using the subdifferential basic chain rule
    \cite[Theorem~10.6]{Rockafellar1998}, the transversality condition
    \eqref{eq:ocp_nonsmooth_transversality} yields:
    \begin{subequations}
      \label{eq:proof_transversality}
      \begin{align}
        \label{eq:proof_transversality_lambda}
        \lambda(t_f) &= \grad_x m[t_f]^\transp\alpha+\grad_xb[t_f]^\transp\beta, \\
        \label{eq:proof_transversality_eta}
        \eta(t_f) &= \alpha\runningku, \\
        \label{eq:proof_transversality_H}
        H[t_f] &= -\grad_t m[t_f]\alpha,
      \end{align}
    \end{subequations}
    for some $\beta\in\reals^{n_b}$. Due to
    \eqref{eq:eta_dynamics}-\eqref{eq:hamiltonian_dynamics},
    \eqref{eq:proof_transversality_eta}-\eqref{eq:proof_transversality_H} and
    absolute continuity, we have \cite[Theorem~9]{Varberg1965}:
    \begin{subequations}
      \label{eq:constant_eta_H}
      \begin{align}
        \eta(t) &= \alpha\zeta,~\forall t\in[0,t_f], \label{eq:eta_constant} \\
        H[t] &= -\grad_t m[t_f]\alpha,~\forall t\in[0,t_f]. \label{eq:hamiltonian_constant}
      \end{align}
    \end{subequations}

    We claim that the primer vector $y(t)\ne 0~\alev{[0,t_f]}$. By
    contradiction, suppose there exists an interval
    $[\tau_1,\tau_2]\subseteq[0,t_f]$ for which
    $y(t)=0$. Condition~\ref{condition:observability} implies that
    $\lambda(\tau_1)=0$. Due to \eqref{eq:constant_eta_H}, this implies
    $\alpha(\ell[\tau_1]+\runningku\sum_{i=1}^M\sigma_i(\tau_1)+\grad_tm[t_f])=0$. Due
    to Condition~\ref{condition:transversality}, it must be that $\alpha=0$
    which implies $(\alpha,\psi(\tau_1))=0$. Since this violates non-triviality
    \eqref{eq:ocp_nonsmooth_non_triviality}, it must be that
    $y(t)\ne 0~\alev{[0,t_f]}$. Having eliminated the pathological case, assume
    $\alpha<0$. In particular, since the necessary conditions in
    Theorem~\ref{theorem:mp} are scale-invariant, we can set $\alpha=-1$ without
    loss of generality. The pointwise maximum condition
    \eqref{eq:ocp_nonsmooth_pointwise_maximum_argmax} implies that the following
    must hold $\alev{[0,t_f]}$: \def\opticmd{\argmax}
    \begin{optimization}{center}{argmax_1}{}{u_i,\gamma_i,\sigma_i}{%
        \begin{array}{l}
          \inlinesum_{i=1}^My(t)\T u_i(t)-\runningku\sigma_i(t)
        \end{array}}{}
      \textnormal{constraints \eqref{eq:rcp_d}-\eqref{eq:rcp_h} hold.}
    \end{optimization}

    We shall now analyze the optimality conditions of \eqref{eq:argmax_1}. For
    concise notation, the time argument $t$ shall be omitted. Expressing
    \eqref{eq:argmax_1} as a minimization and treating constraints
    \eqref{eq:rcp_f} and \eqref{eq:rcp_g} implicitly, we can write the
    Lagrangian of \eqref{eq:argmax_1} \cite{Boyd2004}:
    \begin{multline}
      \label{eq:lagrangian}
      \mathcal L(u_i,\gamma_i,\sigma_i,\lambda_{1\dots 4}^i) = \inlinesum_{i=1}^M\runningku\sigma_i-y\T u_i+
      \lambda_1^i(\norm{u_i}{2}-\sigma_i)+\\
      \lambda_2^i(\gamma_i\rho_1^i-\sigma_i)+
      \lambda_3^i(\sigma_i-\gamma_i\rho_2^i)+{\lambda_4^i}\T C_iu_i,
    \end{multline}
    where $\lambda_j^i\ge 0$ are Lagrange multipliers satisfying the following
    complementarity conditions:
    \begin{subequations}
      \label{eq:complementarity_conditions}
      \begin{align}
        \lambda_1^i(\norm{u_i}{2}-\sigma_i) &= 0, \label{eq:kkt_1} \\
        \lambda_2^i(\gamma_i\rho_1^i-\sigma_i) &= 0, \label{eq:kkt_2} \\
        \lambda_3^i(\sigma_i-\gamma_i\rho_2^i) &= 0, \label{eq:kkt_3} \\
        \lambda_4^i\circ C_iu_i &= 0. \label{eq:kkt_4}
      \end{align}
    \end{subequations}

    Next, the Lagrange dual function is given by:
    \begin{align}
      g(\lambda_{1\dots 4}^i)
      &= \inf_{u_i,\gamma_i,\sigma_i}
        \mathcal L(u_i,\gamma_i,\sigma_i,\lambda_{1\dots 4}^i) \nonumber \\
      &= \inlinesum_{i=1}^M\inf_{\sigma_i}\left[
        (\runningku+\lambda_3^i-\lambda_2^i-\lambda_1^i)\sigma_i\right]-
        \nonumber \\
      &\phantom{=}\hspace{1.4mm} \inlinesum_{i=1}^M\sup_{u_i}\left[(y-C_i\T\lambda_4^i)\T u_i-
        \lambda_1^i\norm{u_i}{2}\right]+ \nonumber \\
      &\phantom{=}\hspace{1.4mm} \textstyle\inf_{\textnormal{\eqref{eq:rcp_f},\eqref{eq:rcp_g}}}\inlinesum_{i=1}^M
        (\lambda_2^i\rho_1^i-\lambda_3^i\rho_2^i)\gamma_i.
        \label{eq:dual_function}
    \end{align}
    
    The dual function bounds the primal optimal cost from above. A non-trivial
    upper-bound requires:
    \begin{subequations}
      \begin{align}
        \norm{y-C_i\T\lambda_4^i}{2}\le\lambda_1^i, \label{eq:dual_function_result_1} \\
        \runningku+\lambda_3^i-\lambda_2^i-\lambda_1^i = 0, \label{eq:dual_function_result_2}
      \end{align}
    \end{subequations}
    where the first inequality is akin to the $\norm{\cdot}{2}$ conjugate
    function \cite[Example~3.26]{Boyd2004}. However, note that if
    \eqref{eq:dual_function_result_1} is strict then $\norm{u_i}{2}=0$ is
    optimal, which is trivially feasible for
    Problem~\ref{problem:ocp}. Substituting \eqref{eq:dual_function_result_2}
    into \eqref{eq:dual_function_result_1} gives the following condition for
    non-trivial solutions:
    \begin{equation}
      \label{eq:prelim_characteristic_equation_1}
      \norm{y-C_i\T\lambda_4^i}{2} = \runningku+\lambda_3^i-\lambda_2^i.
    \end{equation}

    Further simplification is possible by recognizing that a non-trivial
    solution implies $\gamma_i>0$. By Assumption~\ref{ass:norm_bounds},
    \eqref{eq:kkt_2} and \eqref{eq:kkt_3}, this means $\lambda_2^i>0$ and
    $\lambda_3^i>0$ cannot occur simultaneously. Furthemore,
    \eqref{eq:dual_function} reveals that $\gamma_i>0$ is not sub-optimal if and
    only if $\lambda_2^i\rho_1^i-\lambda_3^i\rho_2^i\le 0$. By this reasoning,
    $\lambda_2^i=0$ and $\lambda_3^i\ge 0$ are necessary for optimality. Thus,
    \eqref{eq:prelim_characteristic_equation_1} simplifies to:
    \begin{equation}
      \label{eq:prelim_characteristic_equation_2}
      \norm{y-C_i\T\lambda_4^i}{2} = \runningku+\lambda_3^i.
    \end{equation}

    Next, note that at optimality the left-hand side of
    \eqref{eq:prelim_characteristic_equation_2} equals the Eucledian projection
    onto $\set{U_i}$,
    i.e. $\norm{y-C_i\T\lambda_4^i}{2}=\Proj{\set{U_i}}{y}$. This can be shown
    by contradiction using Assumption~\ref{ass:full_range_cone},
    \eqref{eq:kkt_4} and that
    $u_i=\norm{u_i}{2}(y-C_i\T\lambda_4^i)/\norm{y-C_i\T\lambda_4^i}{2}$ in
    \eqref{eq:dual_function}. Note that the degenerate case of $u_i\ne 0$ and
    $\norm{y-C_i\T\lambda_4^i}{2}=0$ is eliminated in the discussion below which
    leverages Condition~\ref{condition:normality}. Thus
    \eqref{eq:prelim_characteristic_equation_2} simplifies to the following
    relationship, which we call the \textit{characteristic equation} of
    non-trivial solutions to Problem~\ref{problem:rcp}:
    \begin{equation}
      \label{eq:characteristic_equation}
      \Proj{\set{U_i}}{y} = \runningku+\lambda_3^i.
    \end{equation}

    Substituting \eqref{eq:characteristic_equation} into
    \eqref{eq:dual_function} yields:
    \begin{equation}
      \label{eq:dual_function_simplified}
      g(\lambda_{1\dots 4}^i) =
      -\textstyle\sup_{\textnormal{\eqref{eq:rcp_f},\eqref{eq:rcp_g}}}\inlinesum_{i=1}^{K'}
      (\Proj{\set{U_i}}{y}-\runningku)\rho_2^i\gamma_i,
    \end{equation}
    where we assume that the characteristic equation
    \eqref{eq:characteristic_equation} does not hold for $i=K'+1,\dots,M$ such
    that $\gamma_{i>K'}=0$. To facilitate discussion, define the $i$-th input
    \textit{gain} as in \eqref{eq:input_gain}. Note that $\Gamma_i\ge 0$ due to
    \eqref{eq:characteristic_equation}. Thus \eqref{eq:dual_function_simplified}
    becomes:
    \begin{equation}
      \label{eq:dual_function_gain}
      g(\lambda_{1\dots 4}^i) =
      -\textstyle\sup_{\textnormal{\eqref{eq:rcp_f},\eqref{eq:rcp_g}}}
      \inlinesum_{i=1}^{K'}\Gamma_i\gamma_i.
    \end{equation}

    Without loss of generality, assume a descending ordering
    $\Gamma_i\ge\Gamma_j$ for $i>j$. Let $K''\definedas\min\{K,K'\}$. By
    inspection of \eqref{eq:dual_function_gain}, the condition:
    \begin{equation}
      \label{eq:feasibility_condition}
      \Gamma_{K''}>0~\land~\Gamma_{K''}>\Gamma_{K''+1},
    \end{equation}
    is sufficient to ensure that it is optimal to set
    \begin{equation}
      \label{eq:gamma_optimality_structure}
      \gamma_i=
      \begin{cases}
        1 & \text{if }i\le K'', \\
        0 & \text{otherwise.}
      \end{cases}
    \end{equation}

    The lemma holds if \eqref{eq:feasibility_condition} holds
    $\alev{[0,t_f]}$. This is assured by Conditions~\ref{condition:normality}
    and \ref{condition:ambiguity}. Condition~\ref{condition:normality} case (a)
    assures $\Gamma_{K''}>0~\alev{[0,t_f]}$. If on some interval $\Gamma_k=0$,
    Condition~\ref{condition:normality} case (b) assures that $k>K''$. If
    $K''<K$ then due to $\Gamma_{K''}>0$ and the definition of $K'$, it must be
    that $\Gamma_{K''+1}=0\Rightarrow \Gamma_{K''}>\Gamma_{K''+1}$. Else if
    $K''=K$, Condition~\ref{condition:ambiguity} case (a) assures
    $\Gamma_{K}>\Gamma_{K+1}~\alev{[0,t_f]}$. If on some interval
    $\Gamma_k=\Gamma_{k+1}$, Condition~\ref{condition:ambiguity} case (b)
    assures that $k\ne K$.

    


    Thus, \eqref{eq:feasibility_condition} holds $\alev{[0,t_f]}$ and the lemma
    is proved. From \eqref{eq:gamma_optimality_structure}, the structure of the
    optimal solution is \textbf{bang-bang with at most $K$ inputs active}
    $\alev{[0,t_f]}$.
  \end{proof}
\end{lemma}

Lemma~\ref{lemma:lcvx} guarantees that Problem~\ref{problem:rcp} produces a
feasible solution of Problem~\ref{problem:ocp}. We will now show that this
solution is globally optimal, thus proving Theorem~\ref{theorem:lcvx_a}.

\begin{proof}[Proof of Theorem~\ref{theorem:lcvx_a}]
  The solution of Problem~\ref{problem:rcp} is feasible $\alev{[0,t_f]}$ for
  Problem~\ref{problem:ocp} due to Lemma~\ref{lemma:lcvx}. Furthermore, if
  $\runningku=0$ then the cost functions of Problems~\ref{problem:ocp} and
  \ref{problem:rcp} are the same. This is also true when $\runningku=1$ because
  Lemma~\ref{lemma:lcvx} guarantees that $\norm{u_i(t)}{2}=\sigma_i(t)$. The
  optimal costs thus satisfy
  $\optimal{J_{\mathcal O}}\le\optimal{J_{\mathcal R}}$. However, any solution
  of Problem~\ref{problem:ocp} is feasible for Problem~\ref{problem:rcp} by
  setting $\sigma_i(t)=\norm{u_i(t)}{2}$, thus
  $\optimal{J_{\mathcal R}}\le\optimal{J_{\mathcal O}}$. Therefore
  $\optimal{J_{\mathcal R}}=\optimal{J_{\mathcal O}}$ so the
  Problem~\ref{problem:rcp} solution is globally optimal for
  Problem~\ref{problem:ocp}~$\alev{[0,t_f]}$.
\end{proof}

Theorem~\ref{theorem:lcvx_a} implies that Problem~\ref{problem:ocp} is solved in
polynomial time by an SOCP solver applied to Problem~\ref{problem:rcp}. This can
be done efficiently with several numerically reliable SOCP solvers
\cite{Dueri2014}. Therefore the class of \NPhard problems defined by
Problem~\ref{problem:ocp} is in fact of \PP complexity if
$x(t)\in\interior{\set X}$ and
Conditions~\ref{condition:observability}-\ref{condition:transversality} hold.

\subsection{The Case of Active State Constraints}
\label{section:proof_active_state}

So far it has been assumed that the state constraint \eqref{eq:ocp_g} is
inactive. This section guarantees lossless convexification in a limited setting
when \eqref{eq:ocp_g} is activated at a discrete set of times. To begin, define
the \textit{interior time} and \textit{contact time} sets as follows:
\begin{subequations}
  \begin{align}
    \mathcal T_i &\definedas \{t\in (0,t_f): x(t)\in\interior{\mathcal X}\}, \label{eq:interior_times} \\
    \mathcal T_c &\definedas{} [0,t_f]\setminus\mathcal T_i. \label{eq:contact_times}
  \end{align}
\end{subequations}

A point $\tau$ of $\mathcal T_c$ is called an \textit{isolated point} if there
exists a neighborhood of $\tau$ not containing other points of $\mathcal T_c$
\cite{Stein2005}. A set of isolated points is called a \textit{discrete set} and
any discrete subset of a Eucledian space has measure zero
\cite{Acikmese2011}. We can now prove Theorem~\ref{theorem:lcvx_b}.

\begin{proof}[Proof of Theorem~\ref{theorem:lcvx_b}]
  The proof is similar to \cite[Corollary~3]{Acikmese2011}. To begin, let
  $\Sigma_O=\{\optimal{t_f},\optimal{x},\optimal{\xi},\optimal{u_i},\optimal{\gamma_i},\optimal{\sigma_i}\}$
  be the \textit{original solution} returned by Problem~\ref{problem:rcp}, which
  achieves the optimal cost value $\optimal{J_{\mathcal R}}$. Since
  $\mathcal T_c$ is a discrete set, for any consecutive contact times
  $\tau_1< \tau_2$ there exists a large enough real $a>0$ such that
  $\tau_1+1/a<\tau_2-1/a$. Let $\tau_e=\tau_1+1/a$ and $\tau_f=\tau_2-1/a$. Now
  consider solving Problem~\ref{problem:rcp} over $[\tau_e,\tau_e+\Delta\tau]$
  with $t_f=\Delta\tau$, $x_0=x(\tau_e)$, $b[t_f]=x(\Delta\tau)-x(\tau_f)$. Call
  the solution to this problem the \textit{subproblem solution}
  $\Sigma_S=\{\tilde{\Delta}\tau,\tilde{x},\tilde{\xi},\tilde{u}_i,\tilde{\gamma}_i,\tilde{\sigma}_i\}$,
  and let $\optimal{J_S}$ be the achieved optimal cost. We claim that the
  corresponding portion of $\Sigma_O$ must also achieve $\optimal{J_S}$. If it
  does not, the \textit{modified solution}
  $\Sigma_M=\{\hat t_f,\hat{x},\hat{\xi},\hat{u}_i,\hat{\gamma}_i,\hat{\sigma}_i\}$
  such that $\hat t_f=\optimal{t_f}+\tilde{\Delta}\tau-(\tau_f-\tau_e)$ and
  $\{\hat{x},\hat{\xi},\hat{u}_i,\hat{\gamma}_i,\hat{\sigma}_i\}=$
  \begin{equation*}
    \begin{cases}
      \{\optimal{x}(t),\optimal{\xi}(t),\optimal{u}_i(t),
      \optimal{\gamma}_i(t),\optimal{\sigma}_i(t)\}
      & \text{for }t\in [0,\hat t_f]\setminus \\
      & [\tau_e,\tau_e+\tilde{\Delta}\tau], \\
      \{\tilde{x}(t),\tilde{\xi}(t),\tilde{u}_i(t),
      \tilde{\gamma}_i(t),\tilde{\sigma}_i(t)\}
      & \text{for }t\in [\tau_e,\tau_e+\tilde{\Delta}\tau],
    \end{cases}
  \end{equation*}
  is also feasible for Problem~\ref{problem:rcp} and achieves a lower cost than
  $\optimal{J_{\mathcal R}}$, which contradicts that the $[\tau_e,\tau_f]$
  segment of $\Sigma_O$ is optimal. Thus, $\Sigma_S$ must be optimal for the
  original problem. By Theorem~\ref{theorem:lcvx_a}, $\Sigma_S$ must be globally
  optimal for Problem~\ref{problem:ocp}. Since $a$ is arbitrarily large,
  $\Sigma_S$ must be optimal for Problem~\ref{problem:ocp} over
  $t\in (t_1,t_2)$. Let $\mathcal T_c=\{\tau_i,~i=1,2,\dots\}$,
  $\tau_{i}<\tau_{i+1}$ $\forall i$. Hence
  $\interior{\mathcal T_i}=\Union_i(\tau_i,\tau_{i+1})$ and $\Sigma_O$ is
  globally optimal for Problem~\ref{problem:ocp} $\alev{\mathcal T_i}$. Since
  $\mathcal T_c$ is a discrete set, $\closure{\mathcal T_i}=[t_0,t_f]$ and so
  the Problem~\ref{problem:rcp} solution is globally optimal for
  Problem~\ref{problem:ocp} $\alev{[0,t_f]}$.
\end{proof}

\section{Numerical Example}
\label{section:example}

This section shows how rocket landing trajectories can be generated much faster
via Problem~\ref{problem:rcp} than MICP. Python source code for this example is
available online\footnote{\url{https://github.com/dmalyuta/lcvx}}. Consider the
in-plane rocket dynamics:
\begin{equation}
  \label{eq:rocket_dynamics}
  \dot x(t) = A(\omega)x(t)+B\inlinesum_{i=1}^Mu_i(t)+w,
\end{equation}
where the vehicle is treated as a point mass with $x(t)=(r(t),v(t))\in\reals^4$
the position and velocity state and $\omega\in\reals$ the planet rotation rate,
which is assumed to be constant and perpendicular to the trajectory
plane\footnote{This is done for simplicity in order to keep the motion planar. A
  general 3-dimensional angular velocity vector can also be considered.}. The
input $u_i(t)\in\reals^2$ represents an acceleration imparted on the rocket by a
gimballed thruster. The LTI matrices are:
\begin{equation}
  \label{eq:2}
  A(\omega)=
  \begin{bmatrix}
    0 & I \\
    \omega^2 I & 2\omega S
  \end{bmatrix},~
  B =
  \begin{bmatrix}
    0 \\ I
  \end{bmatrix},
  w =
  \begin{bmatrix}
    0 \\
    \omega^2 l+g
  \end{bmatrix},
\end{equation}
where $S=[0~1;-1~0]\in\reals^{2\times 2}$, $I\in\reals^{2\times 2}$ is identity,
$l\in\reals^2$ is the landing pad position with respect to the planet's center
of rotation, and $g\in\reals^2$ is the gravity vector. Note that the dynamics
assume constant mass and gravity for concision, but both can be made variable
within the lossless convexification framework \cite{Acikmese2007,Blackmore2012}.

The rocket is equipped with a single gimballed thruster which operates in two
modes: 1) low-thrust high-gimbal, and 2) high-thrust low-gimbal. A maximum
gimbal angle range of $\theta_i\in (0,\pi)$ is enforced via \eqref{eq:ocp_f} by
setting:
\begin{equation}
  \label{eq:gimbal_constraint}
  C_i =
  \begin{bmatrix}
    -\cos(\theta_i/2) & -\sin(\theta_i/2) \\
    \phantom{-}\cos(\theta_i/2) & -\sin(\theta_i/2)
  \end{bmatrix}.
\end{equation}

We also impose a glide slope constraint as in \cite{Blackmore2010} which
prevents the rocket from approaching the ground too closely prior to touchdown:
\begin{equation}
  \label{eq:glideslope}
  \mathcal X = \{
  x=(r,v)\in\reals^4 : \hat e_y\T r\ge \norm{r}{2}\sin(\gamma_{gs})
  \},
\end{equation}
where $\hat e_y=(0,1)\in\reals^2$ is the unit vector along the altitude axis. We
choose the following parameters, corresponding to a Martian divert maneuver
similar to \cite{Acikmese2007}:
\begin{gather*}
  M = 2,~K=1,~\omega=2\pi/88775~\si{\radian\per\second},~
  \rho_{1}^1=4~\si{\meter\per\second\squared}, \\
  \rho_{1}^2=8~\si{\meter\per\second\squared},~
  \rho_{2}^1=8~\si{\meter\per\second\squared},~
  \rho_{2}^2=12~\si{\meter\per\second\squared},~\theta_1=120~\si{\degree}, \\
  \theta_2=10~\si\degree,~\gamma_{gs}=10~\si{\degree},~
  l=(0,3396.2)~\si{\kilo\meter},~\runningku\in\{0,1\}, \\
  g=(0,-3.71)~\si{\meter\per\second\squared},~
  m[t_f]=(1-\zeta)t_f\xi_{\max}/t_{f,\max}, \\
  \ell(x(t))=10^{-3}\xi_{\max}(|r_1(t)|\tan(\gamma_{gs})+|r_2(t)|)/h_0, \\
  r(0)=(1500,h_0)~\si{\meter},~v(0)=(50,-70)~\si{\meter\per\second}, \\
  r(t_f)=(0,0)~\si{\meter},~v(t_f)=(0,0)~\si{\meter\per\second},
\end{gather*}
where $t_{f,\max}=100~\si{\second}$ is the time of flight upper-bound and
$\xi_{\max}=t_{f,\max}\rho_{2}^2$ is the maximum input integral cost. The
optimal cost is verified to be unimodal in $t_f$ such that golden search can be
applied to find the optimal $t_f$ \cite{Blackmore2010,Kochenderfer2019}. The
initial altitude above ground level (AGL) $h_0$ and $\zeta\in\{0,1\}$ are
independent variables that we shall vary. When $\zeta=0$, we solve for a
minimum-time trajectory, while for $\zeta=1$ we solve for a minimum-fuel
trajectory.

The problem satisfies
Conditions~\ref{condition:observability}-\ref{condition:transversality} under a
few light assumptions. Because the glide slope \eqref{eq:glideslope} maintains
the rocket above zero altitude, $\ell[t]>0~\forall t\in [0,t_f)$ such that
Condition~\ref{condition:transversality} holds irrespective of $m$. To check
Condition~\ref{condition:observability}, recognize that for our choice of
$\ell$:
\begin{equation}
  \label{eq:subdiff_state_penalty_rocket}
  \subdiff\ell[t]\T =
  D\subdiff[r]\ell[t]\T,~D\definedas \begin{bmatrix}
    I \\ 0
  \end{bmatrix}.
\end{equation}

Following the discussion in Section~\ref{subsec:discussion_observability}, we
confirm that the LTI system $\{-A\T,D,B\T,0\}$ is strongly observable, hence
Condition~\ref{condition:observability} holds. To check
Conditions~\ref{condition:normality} and \ref{condition:ambiguity}, we need to
make the following assumption because replacing $\subdiff[r]\ell[t]\T$ with
$\reals^2$ is too conservative.

\begin{assumption}
  \label{ass:non_zero_position_ae}
  The downrange and altitude are non-zero almost everywhere, i.e. $r_1(t)\ne 0$
  and $r_2(t)\ne 0$ $\alev[0,t_f]$.
\end{assumption}

Leveraging Assumption~\ref{ass:non_zero_position_ae} yields a piecewise constant
input to the adjoint system:
\begin{equation}
  \label{eq:grad_r_l_analytic}
  \subdiff[r]\ell[t]\T = \frac{10^{-3}\xi_{\max}}{h_0}\left\{
    \begin{bmatrix}
      \tan(\gamma_{gs}) \\ 1
    \end{bmatrix},
    \begin{bmatrix}
      -\tan(\gamma_{gs}) \\ 1
    \end{bmatrix}\right\}.
\end{equation}

Leveraging \eqref{eq:grad_r_l_analytic}, consider the following LTI system where
a constant input is modelled as a static state, yielding an augmented state
$\lambda'(t)\in\reals^6$:
\begin{subequations}
  \label{eq:adjoint_system_projected}
  \begin{align}
    \dot \lambda'(t)
    &=
      \begin{bmatrix}
        -A\T & D \\ \phantom{-}0 & \phantom{-}0
      \end{bmatrix} \lambda'(t)
             =A'\lambda'(t), \\
    y(t) &= \begin{bmatrix} B\T & 0 \end{bmatrix}\lambda'(t)
                                 =C'\lambda'(t).
  \end{align}
\end{subequations}

\begin{figure}
  \centering
  \begin{subfigure}{1\columnwidth}
    \centering
    \includegraphics[width=0.8\columnwidth]{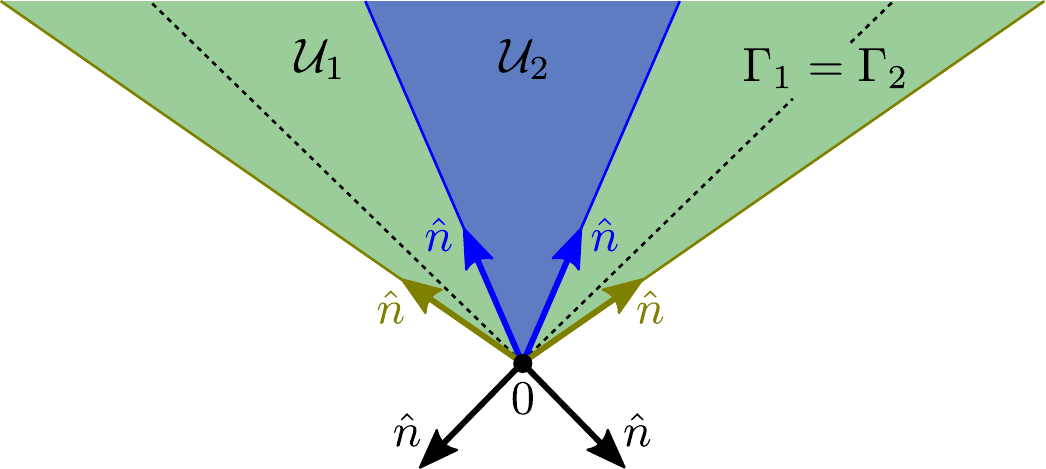}
    \caption{Illustration of the six vectors $\hat n$ that $\dot y(t)$ must not
      be normal to for Conditions~\ref{condition:normality} and
      \ref{condition:ambiguity} to hold.}
    \label{fig:cond_23_normals}
  \end{subfigure}

  \begin{subfigure}{1\columnwidth}
    \centering
    \includegraphics[width=0.8\columnwidth]{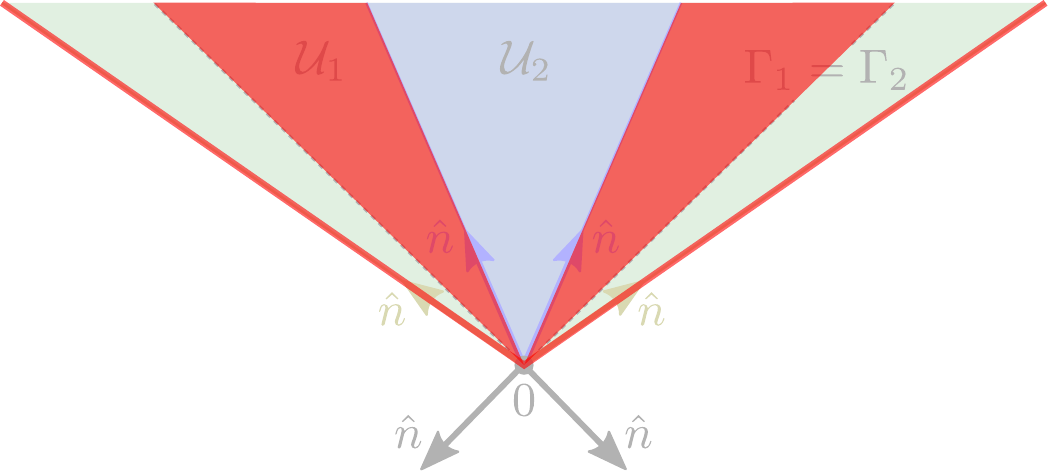}
    \caption{If the normality check fails, the optimal input could point in the
      directions highlighted in red.}
    \label{fig:cond_23_normals_input_evolution}
  \end{subfigure}
  \caption{Illustrated verification of Condition~\ref{condition:normality} and
    \ref{condition:ambiguity} when $\runningku=0$. If $\dot y(t)$ can evolve
    normal to any vector in (\protect\subref{fig:cond_23_normals}), the input
    can point in the directions shown in
    (\protect\subref{fig:cond_23_normals_input_evolution}) while violating
    \eqref{eq:ocp_d}.}
  \label{fig:cond_23}
\end{figure}

When $\runningku=0$, checking Conditions~\ref{condition:normality} and
\ref{condition:ambiguity} reduces to ensuring that $\dot y(t)=C'A' \lambda'(t)$
cannot evolve perpendicular to certain constant vectors $\hat n\in\reals^2$. The
values of $\hat n$ that need to be checked are illustrated in
Figure~\ref{fig:cond_23_normals}. To verify Conditions~\ref{condition:normality}
and \ref{condition:ambiguity}, we check the observability properties of the pair
$\{A',\hat n\T C'A'\}$, as in \cite{Malyuta2019}. Let $V_{\hat n}$ be a matrix
whose columns span the unobservable subspace. It turns out for the rocket
landing problem that $A'V_{\hat n}=0$ $\forall \hat n$.
Conditions~\ref{condition:normality} and \ref{condition:ambiguity} can thus be
violated only by a constant primer vector. If this occurs, the input is
constrained to point in the directions shown in
Figure~\ref{fig:cond_23_normals_input_evolution}. Notice that this constrains
the downrange acceleration to always have the same sign. The following
assumption requires the rocket to experience both acceleration \textit{and}
deceleration. The assumption is satisfied if, for example, the rocket is
initially travelling away from the landing site and has to reverse its velocity.

\begin{assumption}
  \label{ass:acceleration_sign_change}
  The downrange acceleration $\sum_{i=1}^Mu_{i,1}(t)$ changes sign at least once
  over $[0,t_f]$.
\end{assumption}

The assumption is sufficient for Theorem~\ref{theorem:lcvx_a} but not
Theorem~\ref{theorem:lcvx_b}, because a discontinuity in $\dot y(t)$ may occur
at $t\in\set T_c$ \eqref{eq:contact_times} \cite{Hartl1995}. If state
constraints are activated, a ``sufficiently rich'' gimbal history may be assumed
or Conditions~\ref{condition:normality} and \ref{condition:ambiguity} may be
verified \textit{a posteriori}, i.e. the solution is lossless if they hold.

When $\runningku=1$, Condition~\ref{condition:normality} requires
$\norm{y(t)}{2}\ne 1~\alev{[0,t_f]}$. Modal shape analysis for the pair
$\{A',C'\}$ reveals that, given a constant input in
\eqref{eq:grad_r_l_analytic}, $\norm{y(t)}{2}=1$ for an interval is only
possible if $y(t)$ is constant. This is eliminated by
Assumption~\ref{ass:acceleration_sign_change} with the same caveat about state
constraint activation. Checking Condition~\ref{condition:ambiguity} is not
possible \textit{a priori} when $\runningku=1$. The condition is verified
\textit{a posteriori}.

The dynamics \eqref{eq:rocket_dynamics} are discretized via zeroth-order hold on
a uniform temporal grid of $150$ nodes. Python~2.7.15 and ECOS~2.0.7.post1
\cite{Domahidi2013} are used on a Ubuntu~18.04.1 64-bit platform with a 2.5~GHz
Intel Core i5-7200U CPU and 8~GB of RAM. The solution and runtime are compared
to a MICP formulation where \eqref{eq:ocp_d} is implemented directly as a binary
constraint using Gurobi~8.1~\cite{gurobi}.

\begin{figure*}
  \centering
  \begin{subfigure}[b]{0.31\textwidth}
    \centering
    \includegraphics[width=\textwidth]{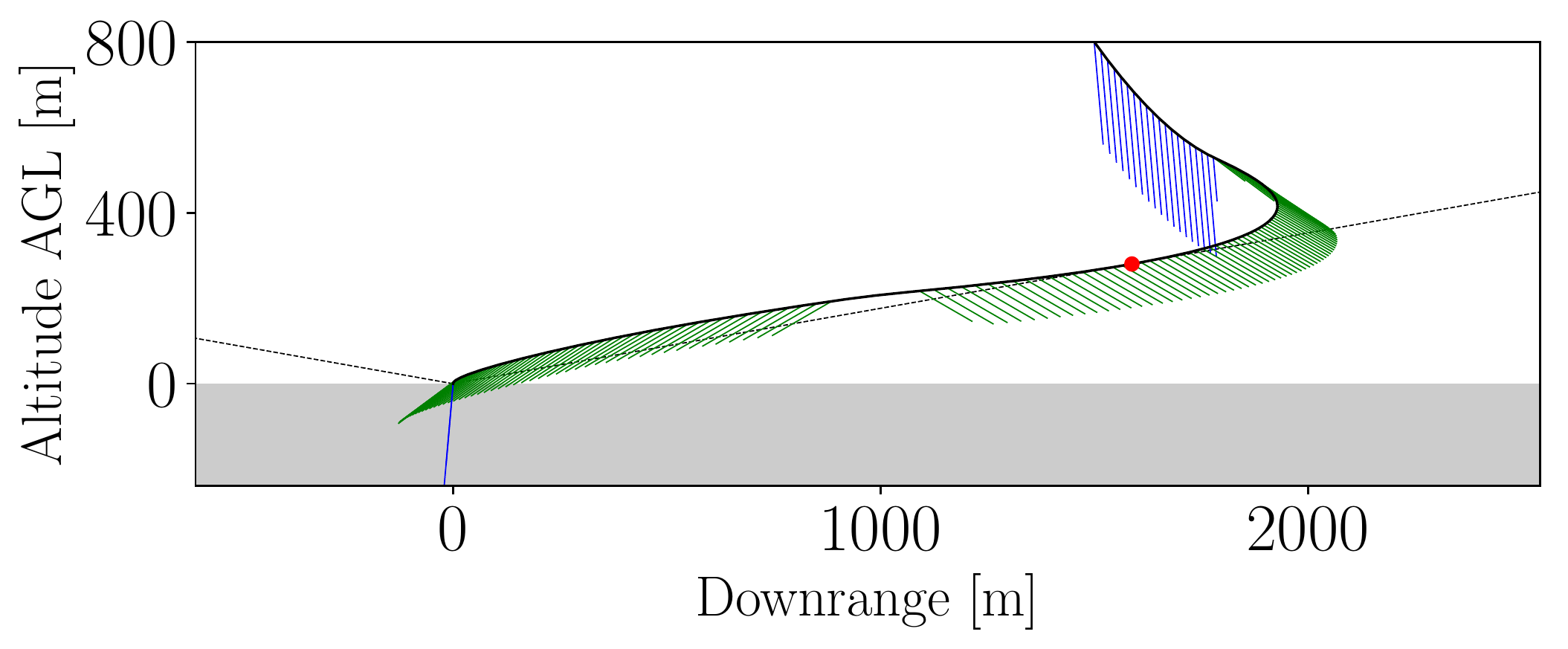}

    \includegraphics[width=\textwidth]{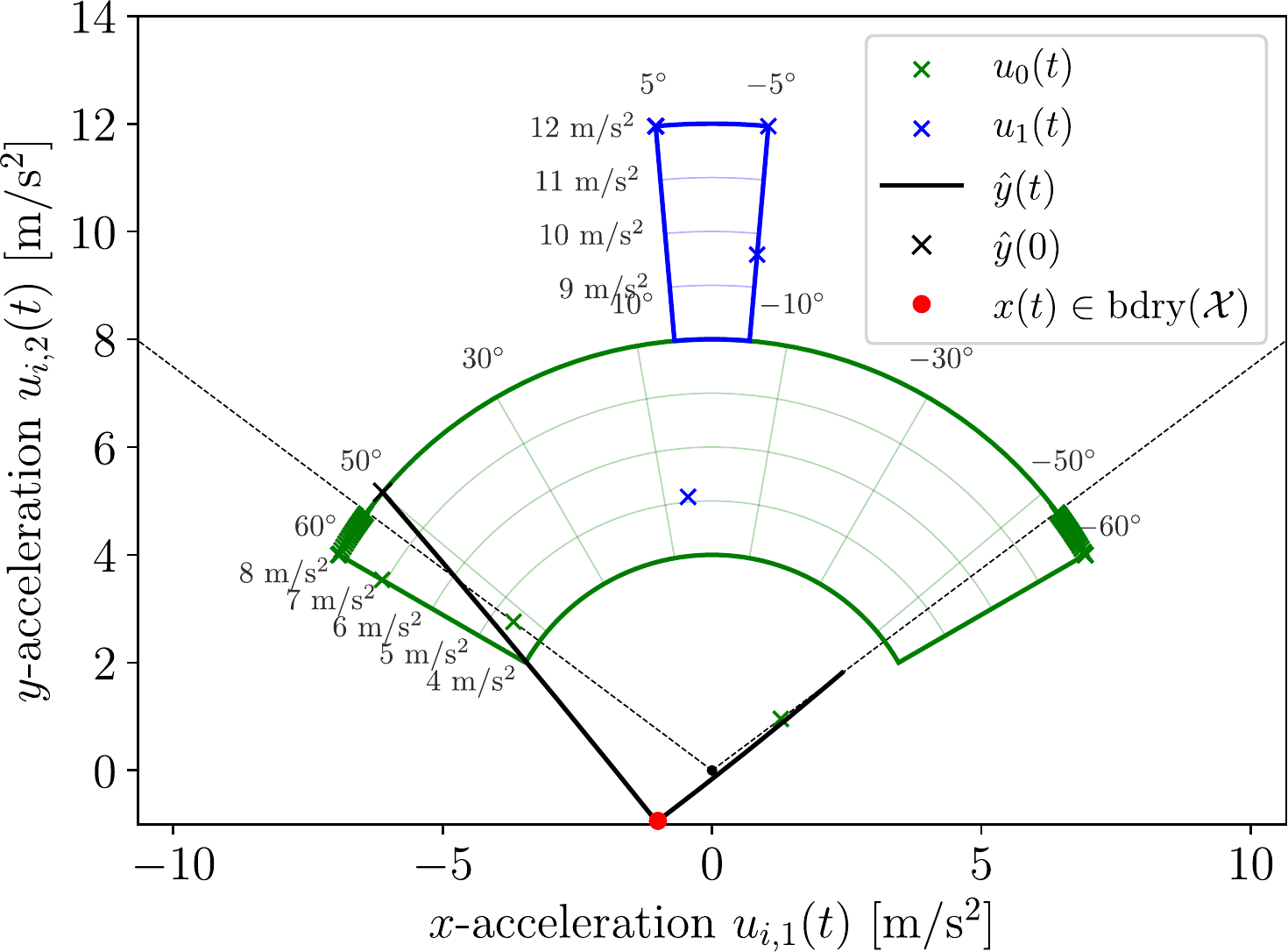}

    \includegraphics[width=\textwidth]{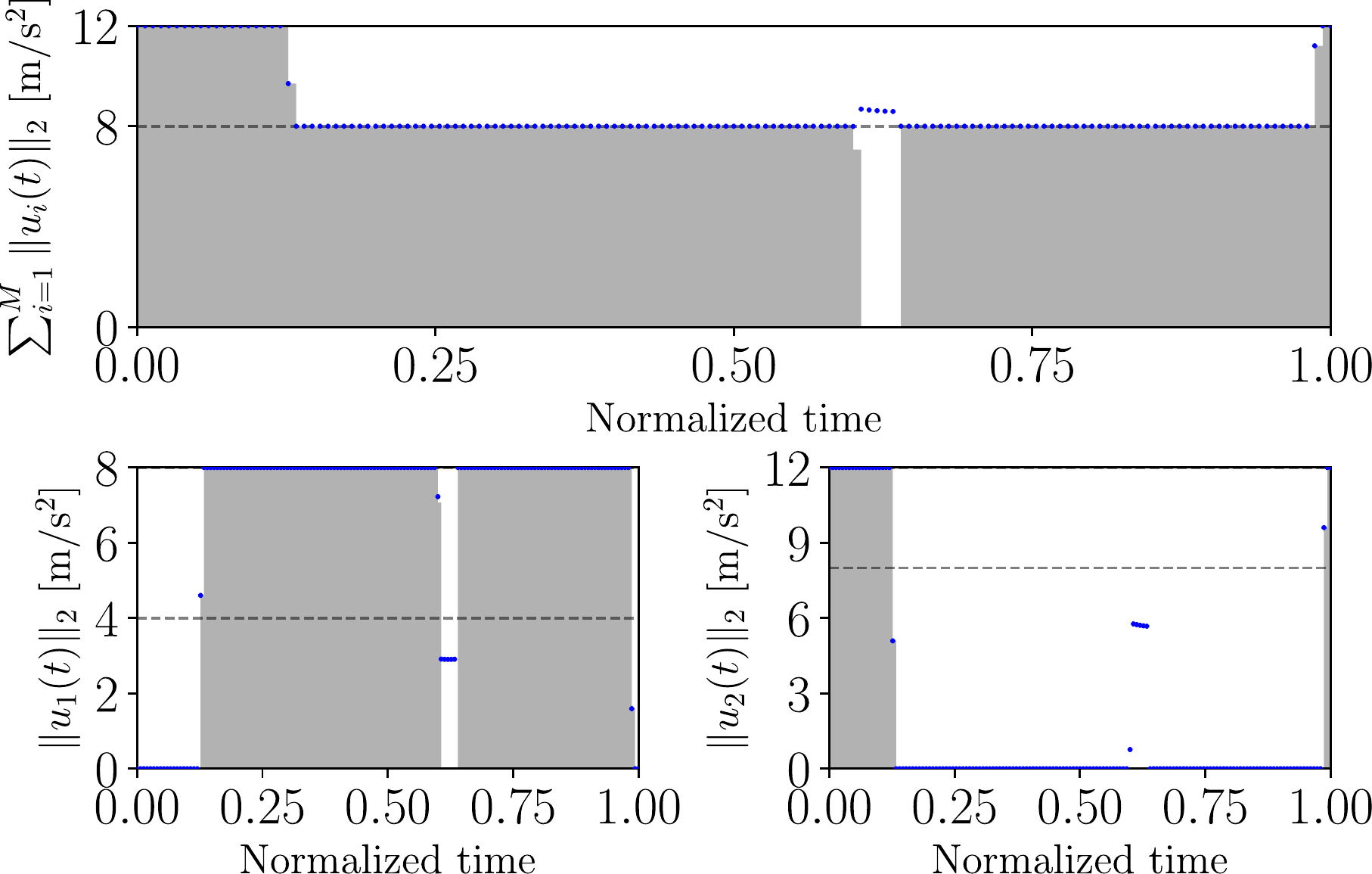}

    \includegraphics[width=\textwidth]{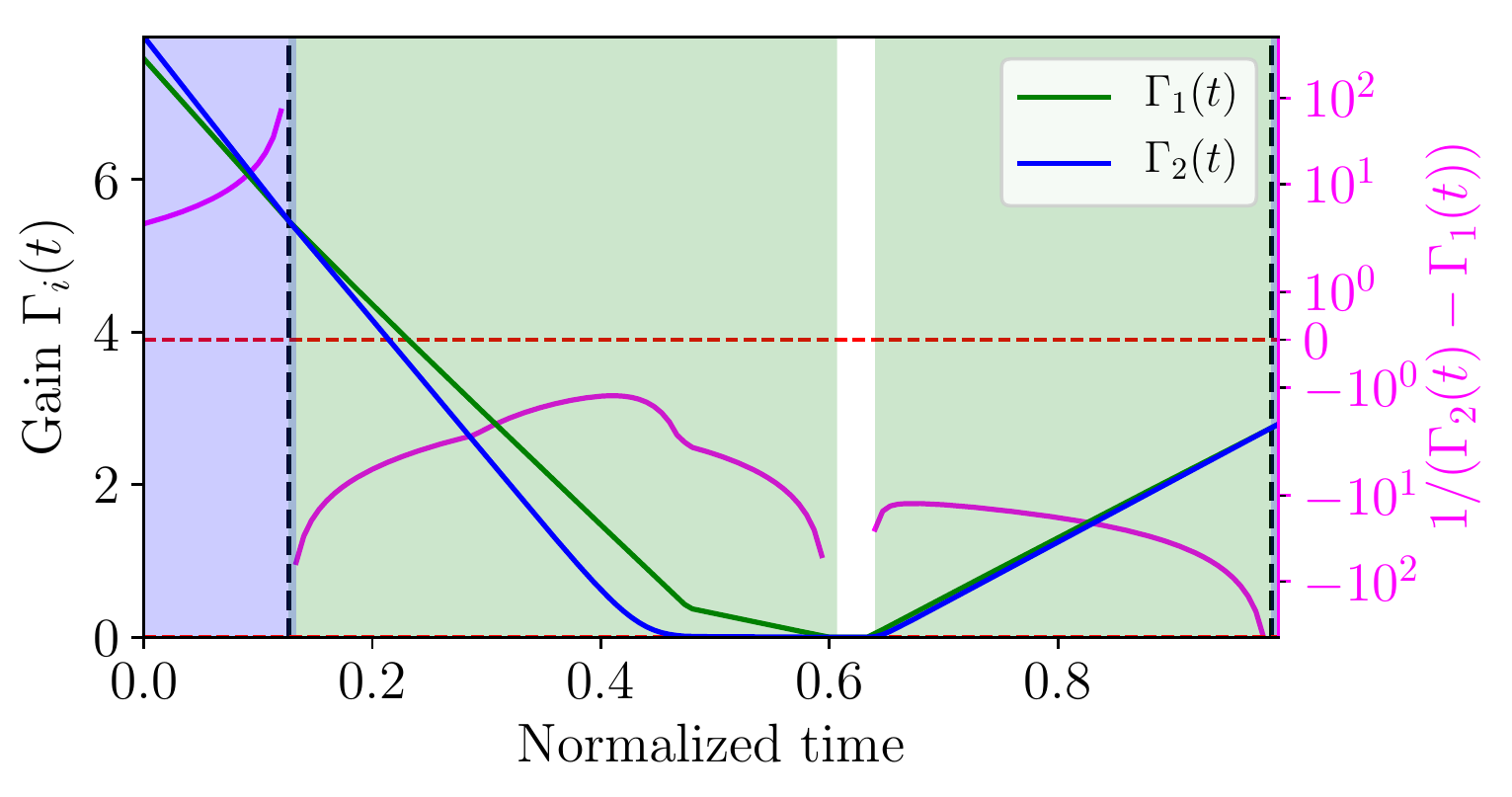}

    \caption{Landing from $h_0=800~\si{\meter}$ AGL, $\runningku=0$. Time of
      flight $t_f=46.93~\si{\second}$.}

    \label{fig:800agl_zeta0}
  \end{subfigure}%
  \hspace{0.005\textwidth}%
  \begin{subfigure}[b]{0.31\textwidth}
    \centering
    \includegraphics[width=\textwidth]{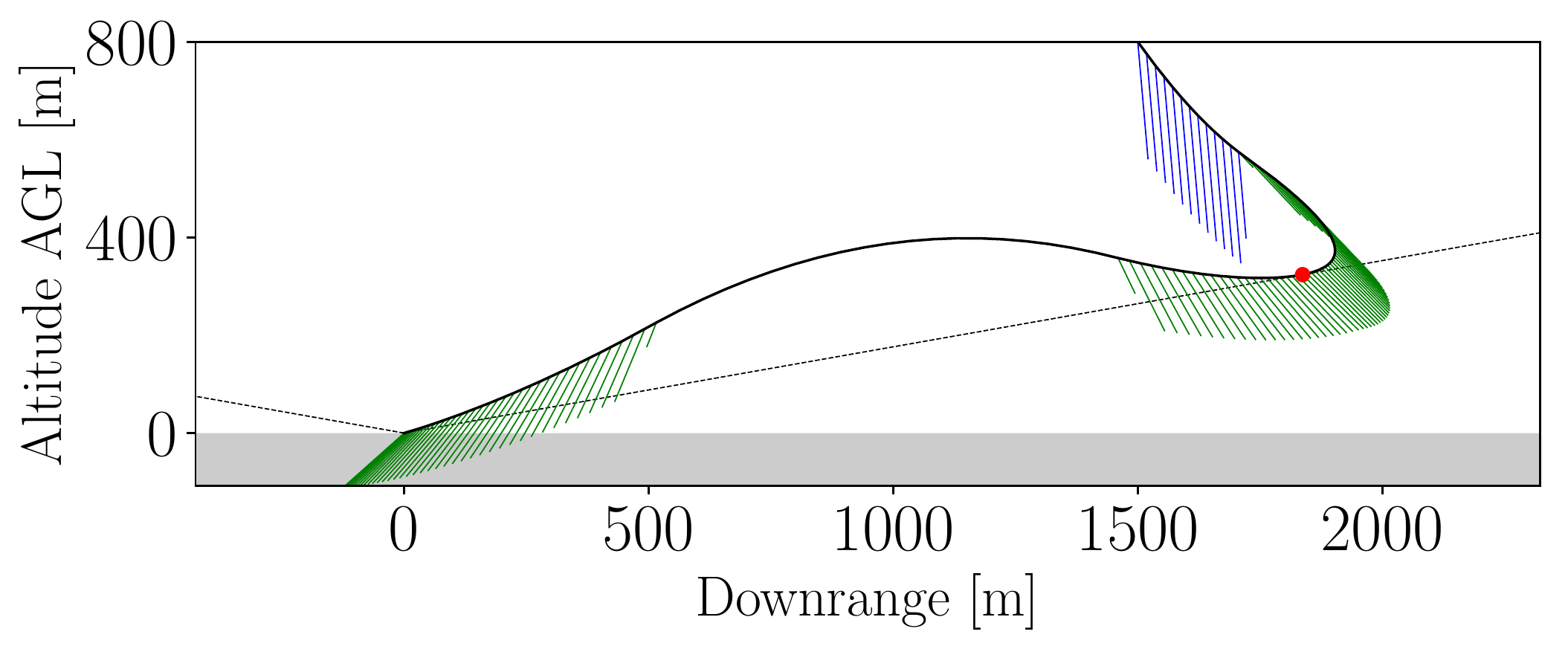}

    \includegraphics[width=\textwidth]{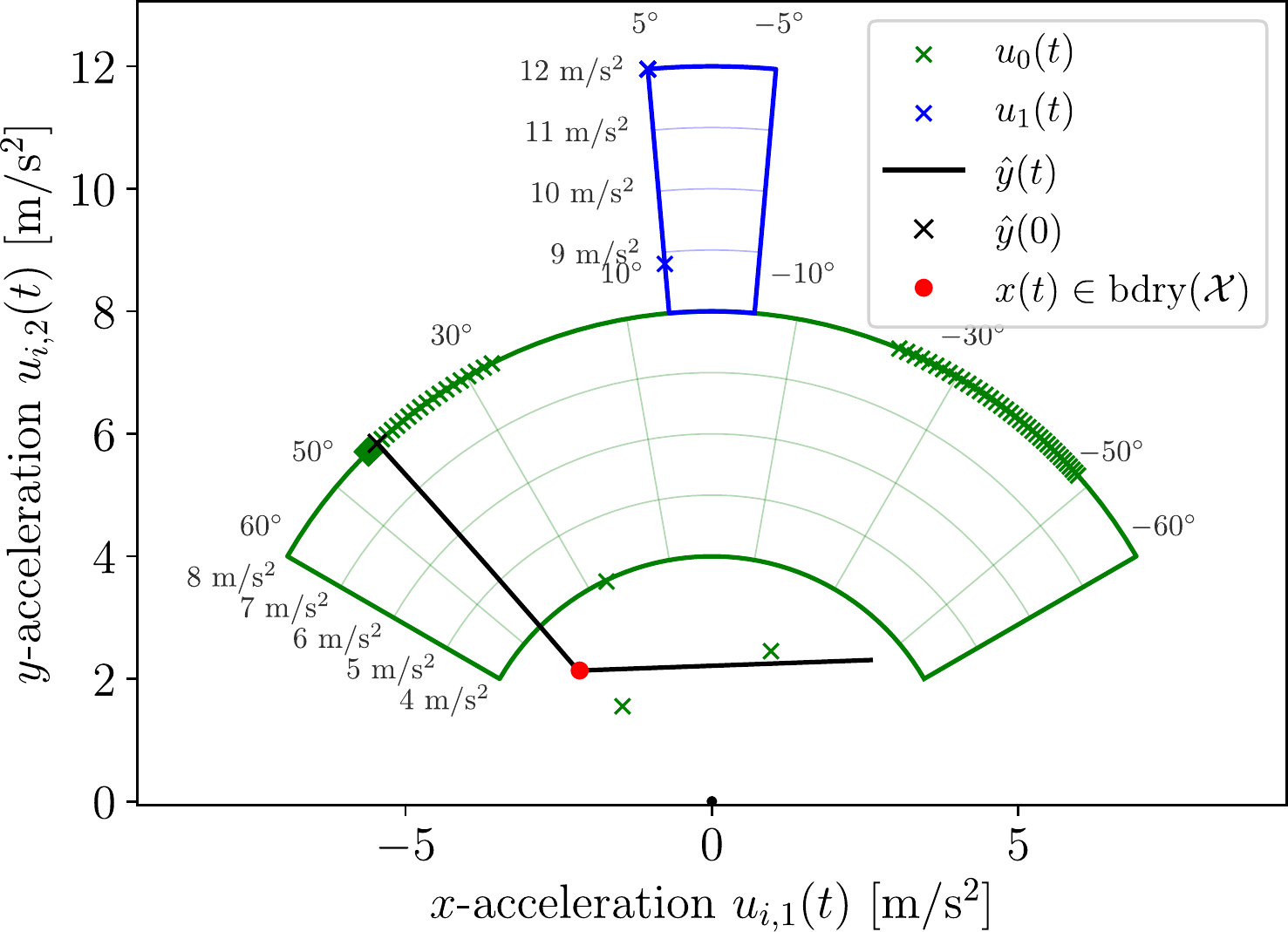}

    \includegraphics[width=\textwidth]{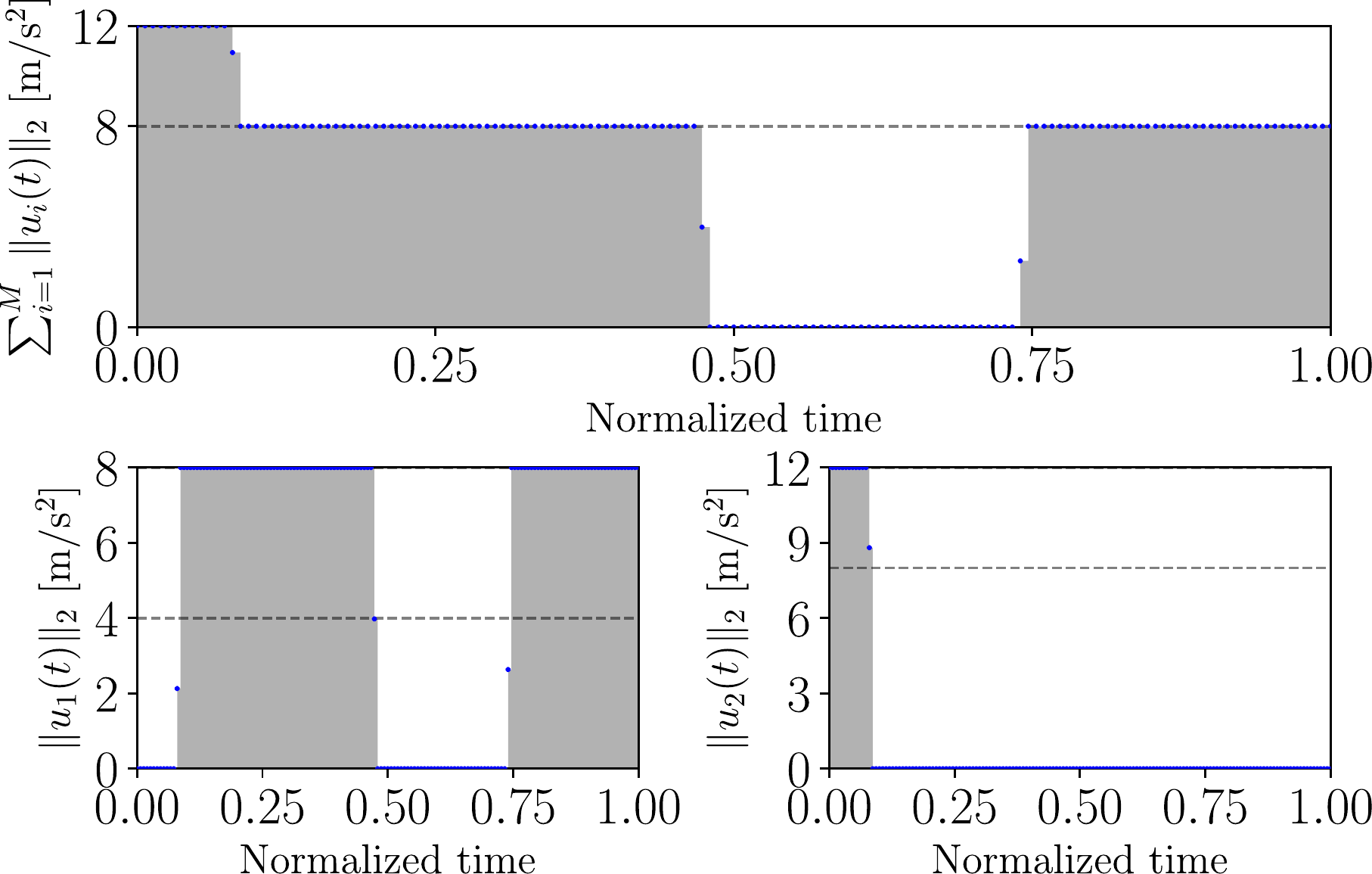}

    \includegraphics[width=\textwidth]{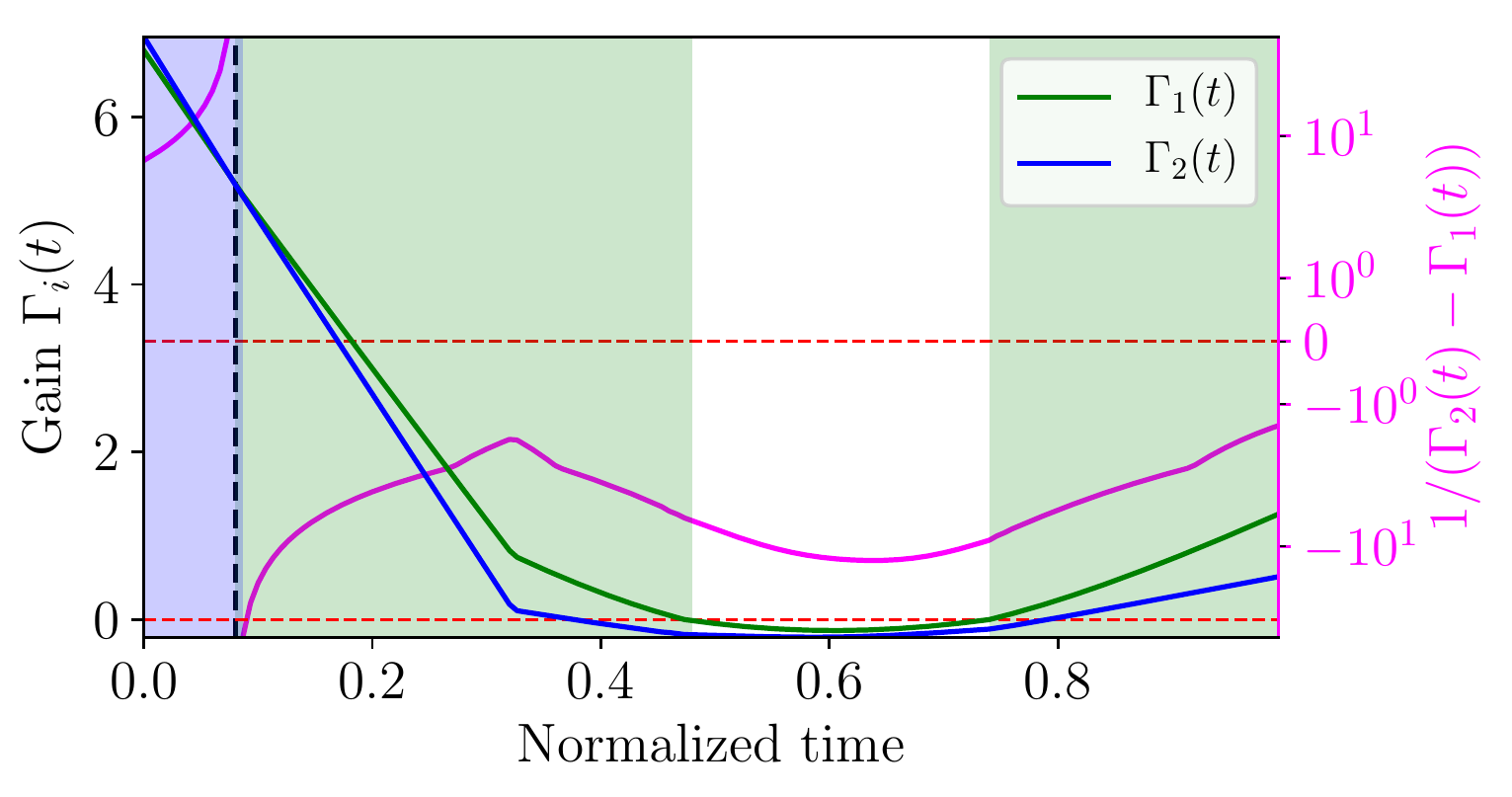}

    \caption{Landing from $h_0=800~\si{\meter}$ AGL, $\runningku=1$. Time of
      flight $t_f=53.97~\si{\second}$.}

    \label{fig:800agl_zeta1}
  \end{subfigure}%
  \hspace{0.005\textwidth}%
  \begin{subfigure}[b]{0.3188\textwidth}
    \centering
    \includegraphics[width=\textwidth]{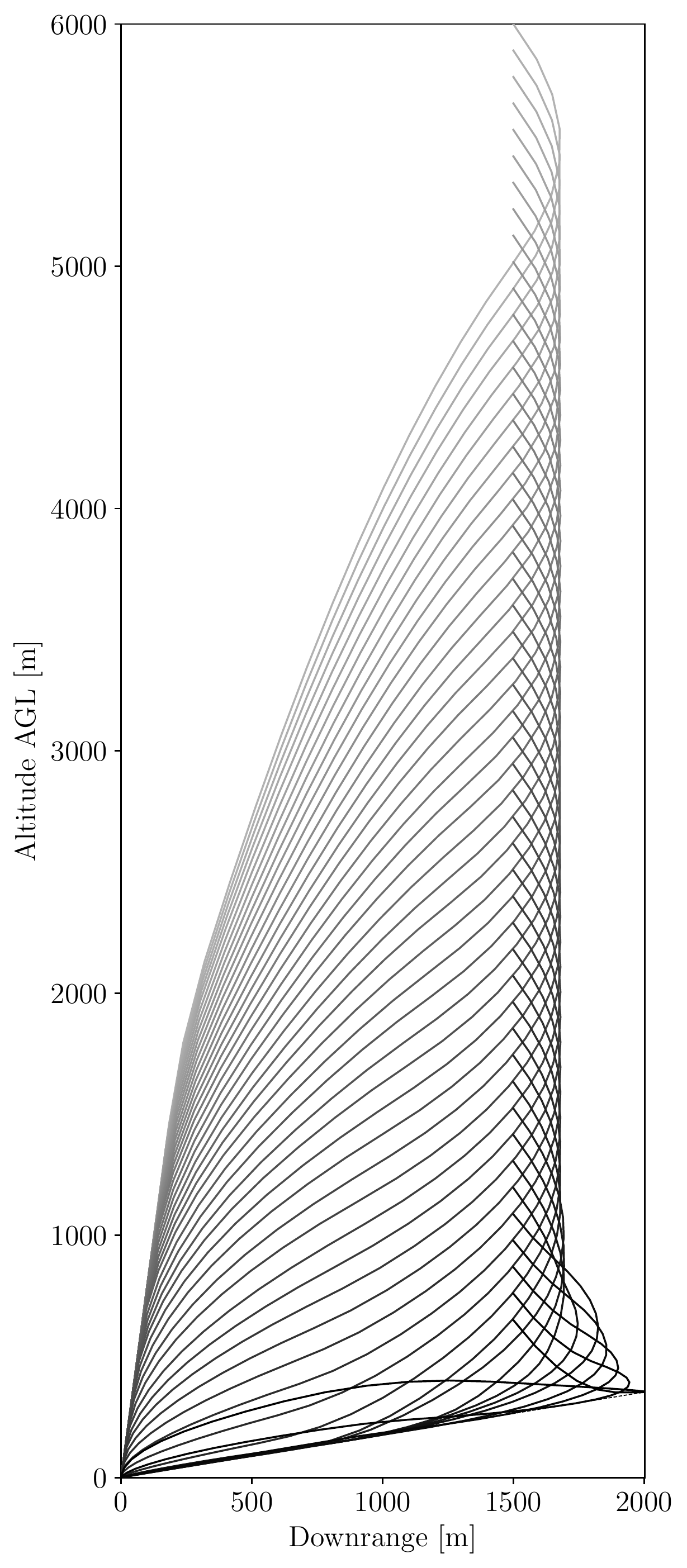}
    \caption{Trajectory sweep over $h_0\in [650,6000]~\si{\meter}$ AGL,
      $\runningku=0$ and $N=30$.}
    \label{fig:sweep}
  \end{subfigure}
  \caption{Landing trajectories computed by Problem~\ref{problem:rcp}. Green
    shows the high-gimbal low-thrust mode and blue shows the low-gimbal
    high-thrust mode. In (\protect\subref{fig:800agl_zeta0}) and
    (\protect\subref{fig:800agl_zeta1}), the top row shows the position
    trajectory with overlaid thrusts ($-u_i(t)$). Dotted lines show glide slope
    \eqref{eq:glideslope}. The second row shows the input with the (normalized)
    primer vector \eqref{eq:primer_vector}. Dotted lines show the equal-gain
    manifold $\Gamma_1(t)=\Gamma_2(t)$. The third row shows the input magnitude
    history. The bottom row shows each input's gain \eqref{eq:input_gain} and
    their difference. The background colour shows when the corresponding input
    is active. In (\protect\subref{fig:sweep}), landing trajectories are shown
    for a sweep over the initial altitude AGL.}
  \label{fig:trajectories}
\end{figure*}

\begin{table}
  \begin{center}
    \captionsetup{width=\columnwidth}
    \caption{Optimal cost and solver runtime when solving
      Problem~\ref{problem:rcp} versus MICP. Dashes show when MICP took too long
      to converge ($>10~\si{\minute}$ per iteration).}
    \label{table:performance_comparison}
    \begin{tabularx}{\columnwidth}{XCCCCC}
      $h_0~[\si{\meter}]$ &
      $\zeta$ &
      $\optimal{J_{\mathcal R}}$\vspace{0.7mm} &
      $\optimal{J_{\textnormal{MICP}}}$\vspace{0.7mm} &
      $t_{\mathcal R}~[\si{\second}]$ &
      $t_{\textnormal{MICP}}~[\si{\second}]$ \\ \hline \hline
      650 & 0 & 636.2 & -- & 2.9 & -- \\
      650 & 1 & 374.5 & -- & 2.4 & -- \\
      800 & 0 & 577.7 & 577.8 & 2.4 & 232.3 \\
      800 & 1 & 350.8 & 350.9 & 2.3 & 269.9 \\
      1000 & 0 & 548.9 & -- & 3.9 & -- \\
      1000 & 1 & 333.7 & 333.7 & 2.3 & 566.8 \\
      1500 & 0 & 493.4 & -- & 2.5 & -- \\
      1500 & 1 & 316.1 & 316.1 & 2.2 & 177.3 \\
      3000 & 0 & 558.0 & 558.0 & 2.5 & 73.1 \\
      3000 & 1 & 323.0 & 323.1 & 1.8 & 505.9 \\ \hline
    \end{tabularx}
  \end{center}
\end{table}

Figure~\ref{fig:trajectories} shows the resulting state, input and input gain
trajectories.  Let us first discuss Figures~\ref{fig:800agl_zeta0} and
\ref{fig:800agl_zeta1}. The top row shows the overall trajectory, from which we
note that Assumptions~\ref{ass:non_zero_position_ae} and
\ref{ass:acceleration_sign_change} are satisfied. The second and third rows show
that the input norm is feasible almost everywhere for
Problem~\ref{problem:ocp}. In particular, the thrust magnitude is bang-bang as
predicted in Lemma~\ref{lemma:lcvx}. The intermediate thrusts occuring at the
rising and falling edges in the third row are discretization artifacts. Recall
that the lossless convexification guarantee is only ``almost everywhere'' in
nature. These artifacts have been observed since the early days of lossless
convexification theory \cite{Acikmese2007}. Note the kink that occurs in the
$y(t)$ trajectory in the second row, which coincides with the glide slope state
constraint activation as highlighted by the red dot in the first row. Looking at
the third row, $\sigma_i(t)\ne\norm{u_i(t)}{2}$ as expected when $\runningku=0$
and both inputs are off, since there is no cost incentive to minimize
$\sigma_i(t)$. Note that optimality nevertheless requires $u_i(t)=0$, as
predicted by Lemma~\ref{lemma:lcvx}. Finally, the fourth row shows the
$\Gamma_i(t)$ trajectories. As predicted by
\eqref{eq:gamma_optimality_structure}, when $\Gamma_i(t)>\Gamma_j(t)$,
optimality forces input $\gamma_i(t)=1$ and $\gamma_j(t)=0$. 

Table~\ref{table:performance_comparison} compares the achieved optimal cost and
solver runtimes of lossless convexification versus a direct MICP implementation
of \eqref{eq:ocp_d}. One can see that the optimal cost values are
quasi-identical, with some slightly lower values for lossless convexification
due to the ``intermediate thrusts'' discussed above. More importantly, solving
Problem~\ref{problem:rcp} is up to two orders of magnitude faster than using
MICP. This is expected, since SOCP has polynomial time complexity in the problem
size while MICP has exponential time complexity. Furthermore, MICP was not able
to find a trajectory in several cases (the computation was aborted when runtime
exceeded 10~\si{\minute} for a single golden search iteration). The third column
of Figure~\ref{fig:trajectories} shows a sequence of 50 landing trajectories for
a sweep over $h_0\in [650,6000]~\si{\meter}$~AGL. Computing this sequence of 50
trajectories with $N = 150$ takes 130~\si{\second}, which is less than the
average MICP solution time for a single trajectory.

\section{Future Work}
\label{section:future_work}

Future work consists of expanding the class of problems that can be
handled. This includes considering different input norm types in
\eqref{eq:ocp_a} and \eqref{eq:ocp_c}, time-varying dynamics in
\eqref{eq:ocp_b}, a lower-bound $L\le\sum_{i=1}^M\gamma_i(t)$ in
\eqref{eq:ocp_e}, a constraint on the input rate of change $\dot u_i(t)$,
persistently active state constraints in \eqref{eq:ocp_g}, and removing the
discretization artifacts observed in Section~\ref{section:example}. A minor
caveat of the Lemma~\ref{lemma:lcvx} proof is that conditions which are proven
to hold ``almost everywhere'' are assumed not to fail on nowhere dense sets of
positive measure (e.g. the fat Cantor set) \cite{Morgan1990}. We do not expect
this pathology to occur for any practical problem, and in the future we seek to
rigorously eliminate this pathology.

\section{Conclusion}
\label{section:conclusion}

This paper presented a lossless convexification solution for a more general
class of optimal control problems with semi-continuous input norms than the one
handled in \cite{Malyuta2019}. By relaxing the problem to a convex one and
proving that the relaxed solution is globally optimal for the original problem,
solutions can be found via convex optimization in polynomial time. The resulting
algorithm is amenable to real-time onboard implementation and can also be used
to accelerate design trade studies.


\bibliography{references}

\end{document}

%% file: latex_common/commands.tex
\newcommand{\union}{\cup}

\newcommand{\Union}{\bigcup}

\newcommand{\set}[1]{\mathcal #1}

\newcommand{\reals}{\mathbb R}
\newcommand{\real}{\mathbb R}

\newcommand{\interior}[1]{\mathrm{int}(#1)}
\newcommand{\closure}[1]{\mathrm{cl}(#1)}

\newcommand{\algvar}[1]{\text{\IfSubStr{#1}{_}{%
    \StrSubstitute{#1}{_}{\textunderscore}}{#1}}}

\newcommand{\grad}{\nabla}

\newcommand{\subdiff}[1][none]{%
  \ifthenelse{\equal{#1}{none}}{%
    \partial%
  }{%
    \partial_{#1}%
  }%
}
\newcommand{\normalcone}[2]{\set{N}_{#1}(#2)}

\DeclareMathOperator{\range}{range}
\DeclareMathOperator{\linspan}{span}
\newcommand{\diag}[2][short]{\mathrm{diag}\ifthenelse{\equal{#1}{short}}{(#2)}{\left(#2\right)}}

\newcommand{\norm}[2]{\|#1\|_{#2}}
\newcommand{\vertiii}[1]{{\left\vert\kern-0.25ex\left\vert\kern-0.25ex\left\vert #1 
        \right\vert\kern-0.25ex\right\vert\kern-0.25ex\right\vert}}

\newcommand{\Proj}[2]{\mathcal P_{#1}(#2)}

\newcommand{\PP}{$\mathcal{P}$\xspace}

\newcommand{\NPhard}{$\mathcal{NP}$-hard\xspace}

\newcommand{\transp}{{\scriptscriptstyle\mathsf{T}}}
\newcommand{\T}{^\transp}

\newcommand{\dd}{\mathrm{d}}

\newcommand{\fun}[2][1]{%
  #2(%
  \foreach \index in {1, ..., #1} {%
    \ifthenelse{\equal{\index}{#1}}{%
      \cdot%
    }{%
      \cdot,%
    }%
  })}
\newcommand{\definedas}[1][tri]{%
  \ifthenelse{\equal{#1}{tri}}{\triangleq}{\coloneqq}}

\newcommand{\inlinesum}{\textstyle\sum}

\DeclareMathOperator*{\exptx}{exp}
\renewcommand{\exp}[2][exponent]{\ifthenelse{\equal{#1}{exponent}}{e^{#2}}{\exptx\left(#2\right)}}
\DeclareMathOperator*{\expsf}{\mathsf{exp}}
\newcommand{\expm}[2][exponent]{\ifthenelse{\equal{#1}{exponent}}{\mathsf{e}^{#2}}{\expsf\left(#2\right)}}

\makeatletter
\DeclareFontFamily{U}{tipa}{}
\DeclareFontShape{U}{tipa}{m}{n}{<->tipa10}{}
\newcommand{\arc@char}{{\usefont{U}{tipa}{m}{n}\symbol{62}}}%
\renewcommand{\arc}[1]{\mathpalette\arc@arc{#1}}
\newcommand{\arc@arc}[2]{%
  \sbox0{$\m@th#1#2$}%
  \vbox{
    \hbox{\resizebox{\wd0}{\height}{\arc@char}}
    \nointerlineskip
    \box0
  }%
}
\makeatother

\newcommand{\alev}[1]{\text{a.e. }#1}
\newcommand{\ev}[3][c]{%
  \ifthenelse{\equal{#1}{c}}{%
    \ifthenelse{\equal{#2}{}}{\forall[0,#3]}{\forall[#2,#3]}%
  }{%
    \ifthenelse{\equal{#1}{o}}{%
      \ifthenelse{\equal{#2}{}}{\forall(0,#3)}{\forall(#2,#3)}%
    }{%
      \ifthenelse{\equal{#1}{oc}}{%
        \ifthenelse{\equal{#2}{}}{\forall(0,#3]}{\forall(#2,#3]}%
      }{%
        \ifthenelse{\equal{#2}{}}{\forall[0,#3)}{\forall[#2,#3)}%
      }%
    }%
  }%
}

\DeclareMathOperator*{\argmax}{argmax}
\DeclareMathOperator*{\argmin}{argmin}

\def\opticmd{\min}

\newcommand{\optimal}[1]{#1^*}

\newcounter{l}
\newcounter{j}
\newcounter{k}

\makeatletter
\newenvironment{taggedsubequations}[1]
{%
  \addtocounter{equation}{-1}%
  \begin{subequations}%
    \def\@currentlabel{#1}%
    \renewcommand{\theequation}{#1.\alph{equation}}%
  }
  {\end{subequations}}
\makeatother 

\newcommand{\lloptimization}[7][center]{
  \setsepchar{\#}%
  \readlist\mylist{#6}
  \ifthenelse{\equal{#7}{}}{\begin{subequations}}{\begin{taggedsubequations}{#7}}
    \ifthenelse{\equal{#2}{}}{}{\label{eq:#2}}
    \ifthenelse{\equal{#1}{left}}{
      \begin{flalign}
        \ifthenelse{\equal{#2}{}}{}{\label{eq:#2_a}}
        \ifthenelse{\equal{#3}{}}{}{#3 = }
        &\opticmd_{\ifthenelse{\equal{#4}{}}{}{#4}}~ #5~\mathrm{s.t.}\hspace{-1mm} &&\\
        \forloop{k}{0}{\arabic{k} < \listlen\mylist[]}{
          \setcounter{j}{\value{k}+1}
          \setcounter{l}{\value{k}+2}
          \ifthenelse{\equal{#2}{}}{}{\label{eq:#2_\alph{l}}}
          \ifthenelse{\equal{\arabic{j}}{\listlen\mylist[]}}{%
            &\mylist[\arabic{j}]%
          }{%
            &\mylist[\arabic{j}] &&\\%
          }%
        }
      \end{flalign}
    }{\ifthenelse{\equal{#1}{center}}{
      \begin{align}
        \ifthenelse{\equal{#2}{}}{}{\label{eq:#2_a}}
        \ifthenelse{\equal{#3}{}}{}{#3 = }
        &\opticmd_{\ifthenelse{\equal{#4}{}}{}{#4}}~ #5~\mathrm{s.t.}\hspace{-1mm} &&\\
        \forloop{k}{0}{\arabic{k} < \listlen\mylist[]}{
          \setcounter{j}{\value{k}+1}
          \setcounter{l}{\value{k}+2}
          \ifthenelse{\equal{#2}{}}{}{\label{eq:#2_\alph{l}}}
          \ifthenelse{\equal{\arabic{j}}{\listlen\mylist[]}}{%
            &\mylist[\arabic{j}]%
          }{%
            &\mylist[\arabic{j}]\\%
          }%
        }
      \end{align}}{\ifthenelse{\equal{#1}{left*}}{
      \begin{flalign*}
        \ifthenelse{\equal{#3}{}}{}{#3 = }
        &\opticmd_{\ifthenelse{\equal{#4}{}}{}{#4}}~ #5~\mathrm{s.t.}\hspace{-1mm} &&\\
        \forloop{k}{0}{\arabic{k} < \listlen\mylist[]}{
          \setcounter{j}{\value{k}+1}
          \setcounter{l}{\value{k}+2}
          \ifthenelse{\equal{\arabic{j}}{\listlen\mylist[]}}{%
            &\mylist[\arabic{j}]%
          }{%
            &\mylist[\arabic{j}] &&\\%
          }%
        }
      \end{flalign*}
    }{\begin{align*}
        \ifthenelse{\equal{#3}{}}{}{#3 = }
        &\opticmd_{\ifthenelse{\equal{#4}{}}{}{#4}}~ #5~\mathrm{s.t.}\hspace{-1mm} &&\\
        \forloop{k}{0}{\arabic{k} < \listlen\mylist[]}{
          \setcounter{j}{\value{k}+1}
          \setcounter{l}{\value{k}+2}
          \ifthenelse{\equal{\arabic{j}}{\listlen\mylist[]}}{%
            &\mylist[\arabic{j}]%
          }{%
            &\mylist[\arabic{j}]\\%
          }%
        }
      \end{align*}}}
    }
  \ifthenelse{\equal{#7}{}}{\end{subequations}}{\end{taggedsubequations}}
}

\newcommand{\llpoptimization}[8][left]{
  \begin{problem}[#8]\ifthenelse{\equal{#2}{NoNe}}{}{\textbf{#2.}}%
    \ifthenelse{\equal{#3}{}}{}{\label{problem:#3}}%
    \lloptimization[#1]{#3}{#4}{#5}{#6}{#7}{#8}%
  \end{problem}%
}

\NewEnviron{optimization}[6]{\lloptimization[#1]{#2}{#3}{#4}{#5}{\BODY}{#6}}
\NewEnviron{foptimization}[7][NoNe]{%
  \begin{mdframed}[default,linewidth=1pt]%
    \llpoptimization[#2]{#1}{#3}{#4}{#5}{#6}{\BODY}{#7}%
  \end{mdframed}%
}
\NewEnviron{poptimization}[7][NoNe]{%
  \llpoptimization[#2]{#1}{#3}{#4}{#5}{#6}{\BODY}{#7}%
}

\definecolor{darkolivegreen}{rgb}{0.33, 0.6, 0.18}
\definecolor{green}{rgb}{0, 0.5, 0}
\definecolor{orange}{rgb}{1, 0.4, 0}
\definecolor{darkgray}{rgb}{0.2, 0.2, 0.2}

\newcommand{\postmeeting}[2][show]{%
  \ifthenelse{\equal{#1}{show}}{%
    {\color{orange} #2}%
  }{}%
}
\newcommand{\idea}[2][show]{%
  \ifthenelse{\equal{#1}{show}}{%
    {\color{orange} #2}%
  }{}%
}
\newcommand{\todo}[2][show]{%
  \ifthenelse{\equal{#1}{show}}{%
    {\color{blue}%
      \ifthenelse{\equal{#2}{}}{TODO}{(TODO: #2)}}%
  }{}%
}
\newcommand{\question}[2][show]{%
  \ifthenelse{\equal{#1}{show}}{%
    {\color{darkolivegreen} (Q: #2)}%
  }{}%
}
\newcommand{\fixme}[2][show]{%
  \ifthenelse{\equal{#1}{show}}{%
    {\color{red} (FIXME: #2)}%
  }{}%
}
\newcommand{\mycomment}[2][show]{%
  \ifthenelse{\equal{#1}{show}}{%
    {\color{red} (C: #2)}%
  }{}%
}

\newcommand{\Behcet}{Beh\c{c}et\xspace}
\newcommand{\Acikmese}{A\c{c}{\i}kme\c{s}e\xspace}

\makeatletter
\@ifclassloaded{beamer}{}{}
\makeatother

\renewcommand{\algref}[3]{\ifthenelse{\equal{#2}{}}{Algorithm~\ref{alg:#1}\xspace}%
  {\ifthenelse{\equal{#3}{}}{line~\ref{alg:#1:line:#2} %
      of Algorithm~\ref{alg:#1}\xspace}{lines~%
      \ref{alg:#1:line:#2}-\ref{alg:#1:line:#3} %
      of Algorithm~\ref{alg:#1}\xspace}}}

%% file: latex_common/style.tex
\hypersetup{
  colorlinks,
  linkcolor={red},
  citecolor={green},
  urlcolor={blue}
}

\allowdisplaybreaks

\tcbset{highlight math style={enhanced,
colframe=red!60!black,colback=yellow!50!white,arc=4pt,boxrule=1pt,
drop fuzzy shadow}}

\newcolumntype{C}{>{\centering\arraybackslash}X}
\newcolumntype{L}{>{\raggedright\arraybackslash}X}
\newcolumntype{R}{>{\raggedleft\arraybackslash}X}

\makeatletter
\@ifclassloaded{beamer}{}{

\newcounter{theorem}
\newtheorem{innertheorem}{Theorem}
\makeatletter
\newenvironment{theorem}[1]
{\refstepcounter{theorem}%
\protected@edef\@currentlabelname{\value{theorem}}%
\renewcommand\theinnertheorem{{\thetheorem#1}}
\innertheorem}
{\endinnertheorem}
\makeatother

\newtheorem{lemma}{Lemma}

\theoremstyle{definition}

\newcounter{problem}

\newtheorem{innerproblem}{Problem}
\makeatletter
\newenvironment{problem}[1][]
{\ifthenelse{\equal{#1}{}}
{\refstepcounter{problem}%
\protected@edef\@currentlabelname{\value{problem}}%
\renewcommand\theinnerproblem{\theproblem}\innerproblem}
{\renewcommand\theinnerproblem{#1}\innerproblem}}
{\endinnerproblem}
\makeatother

\newtheorem{definition}{Definition}

\newcounter{condition}
\newtheorem{innercondition}{Condition}
\makeatletter
\newenvironment{condition}[2][]
{\refstepcounter{condition}%
\protected@edef\@currentlabelname{\value{condition}}%
\ifthenelse{\equal{#2}{}}{\renewcommand\theinnercondition{{\thecondition#1}}}
{\renewcommand\theinnercondition{{\thecondition#1 \textnormal{(#2)}}}}
\innercondition}
{\endinnercondition}
\makeatother

\newtheorem{assumption}{Assumption}
}
\makeatother

\theoremstyle{empty}

\declaretheoremstyle[
  headfont=\color{black}\normalfont\bfseries,
  bodyfont=\color{darkgray}\normalfont,
]{colored}

\newmdenv[
outerlinewidth = 1,%
linewidth = 0pt,%
roundcorner = 2pt,%
leftmargin = 20,%
rightmargin = 0,%
backgroundcolor = lightgray!10,%
outerlinecolor = gray!50,%
innertopmargin = \topskip,%
splittopskip = \topskip,%
frametitle = Summary,%
frametitlebelowskip = 0pt,%
]{summary_box}

\newmdenv[
outerlinewidth = 1,%
linewidth = 0pt,%
roundcorner = 2pt,%
leftmargin = 60,%
rightmargin = 60,%
backgroundcolor = yellow!40,%
outerlinecolor = black!50,%
innertopmargin = \topskip,%
splittopskip = \topskip,%
]{highlight_box}

\newif\ifnomentry

\renewcommand{\nomgroup}[1]{%
\nomentryfalse
\ifthenelse{\equal{#1}{V}}{\item[\textbf{Variables}]}{%
\ifthenelse{\equal{#1}{A}}{\item[\textbf{Abbreviations}]}{%
\ifthenelse{\equal{#1}{E}}{\item[\textbf{Equations}]}{}}}
\nomentrytrue
}
\newcommand{\defvar}[3][show]{\nomenclature[V]{#2}{#3}\ifthenelse{\equal{#1}{show}}{#2}{}}

\algrenewcommand\algorithmicindent{1em}
\algrenewcommand\alglinenumber[1]{{\sf\footnotesize#1}}
\algtext*{EndWhile}
\algtext*{EndIf}
\algtext*{EndFor}

\algrenewcommand\algorithmicrequire{\textbf{Precondition:}}
\algrenewcommand\algorithmicensure{\textbf{Postcondition:}}
\algnewcommand\algorithmicinput{\textbf{Input:}}
\algnewcommand\Input{\item[\algorithmicinput]}
\algnewcommand\algorithmicfunctions{\textbf{Lambda Functions:}}
\algnewcommand\Functions{\item[\algorithmicfunctions]}
\algnewcommand\Stop[1][]{\State \textbf{STOP} {\color{green} \ul{#1}}}
\algnewcommand\Break{\State \textbf{break}}
\algnewcommand\Continue{\State \textbf{continue}\xspace}
\algnewcommand\Error[1][]{\State \textbf{error} {\color{red} \uwave{#1}}}
\algrenewcommand\algorithmiccomment[1]{\hfill {\color{gray} \(\triangleright\) #1}}
\makeatletter
\algnewcommand\Section[1]{\Statex \hskip\ALG@thistlm \textbf{\color{gray} \(\triangleright\) #1}}
\makeatother
\algdef{SE}[PARAMETERS]{Parameters}{EndParameters}
   {\algorithmicparameters}
   {\algorithmicend\ \algorithmicparameters}
\algnewcommand{\algorithmicparameters}{\textbf{parameters}}
\newcommand{\algorithmiclambda}[3][]{\ifthenelse{\equal{#1}{}}{\texttt{#2}(#3)}{$\texttt{#2}\gets\textbf{lambda}~#3:~#1$}}

%% file: special.tex
\newcommand{\runningku}{\zeta}
\newcommand{\tzero}{0}
